\documentclass[12pt]{article}
\usepackage{lmodern}
\usepackage[T1]{fontenc}
\usepackage{amsmath,amssymb,amsfonts,amsopn,amsthm,amsbsy}
\usepackage{latexsym}
\usepackage{color}
\usepackage{cite}  
\usepackage{xfrac}
\usepackage{bbm}   
\usepackage[a4paper, left=2.6cm, right=2.6cm, top=2.4cm, bottom=2.4cm]{geometry}
\DeclareMathOperator{\tr}{tr\,}

\newcommand{\nn}{\nonumber}

\newcommand{\ths}{\theta^{\star}}

\newcommand{\ve}{\varepsilon}
\newcommand{\R}{{\mathbb R}}
\newcommand{\D}{{\mathbb C}}
\newcommand{\T}{{\mathbb T}}
\newcommand{\PT}{{P\mathbb{T}}}
\newcommand{\TC}{{\mathcal T}}
\newcommand{\pis}[1]{\mathcal{#1}}

\newcommand{\di}{{\mathrm d}}
\renewcommand{\S}{{\cal S}^3}
\renewcommand{\SS}{{\cal S}^3_{\mathrm{dev}}}
\newcommand{\eps}{\varepsilon}
\newcommand{\epsp}{\varepsilon^p}
\newcommand{\epsps}{\varepsilon^\star}
\newcommand{\Tl}{\T^\lambda}
\newcommand{\epspl}{\varepsilon^{p,\lambda}}
\newcommand{\thl}{\theta^\lambda}
\newcommand{\tht}{\tilde{\theta}}
\newcommand{\ul}{u^\lambda}

\newcommand{\Ho}{H^1(\Omega)}
\newcommand{\HoR}{H^1(\Omega;\R^3)}
\renewcommand{\r}{_{\Delta}}
\newcommand{\id}{\mathbbm{1}}
\newcommand{\comment}[1]{\marginpar{\flushleft{\textsf{\scriptsize{#1}}}}}
\newcommand{\dyw}{\operatorname{div}_x}
\newcommand{\graph}{\operatorname{graph}\,}

\newcommand{\norm}[1]{\big\|#1\big\|}
\renewcommand{\leq}{\leqslant}
\renewcommand{\geq}{\geqslant}
\title{\textsc{On renormalised solution for thermomechanical problem in perfect-plasticity}}
\author{\textbf{Leszek Bartczak}\\[0.5ex] 
\textbf{\footnotesize{Faculty of Mathematics and
 Information Science, Warsaw University of Technology,}}\\[-1ex]
\textbf{\footnotesize{ul. Koszykowa 75, 00-662 Warsaw, Poland}}\\[-1ex]
\textbf{\footnotesize{E-Mail: l.bartczak@mini.pw.edu.pl}}\\[2ex]
\textbf{Sebastian Owczarek}\\[0.5ex] 
\textbf{\footnotesize{Faculty of Mathematics and
 Information Science, Warsaw University of Technology,}}\\[-1ex]
\textbf{\footnotesize{ul. Koszykowa 75, 00-662 Warsaw, Poland}}\\[-1ex]
\textbf{\footnotesize{E-Mail: s.owczarek@mini.pw.edu.pl}}}
\date{}

\newtheorem{tw}{Theorem}[section]
\newtheorem{lem}[tw]{Lemma}
\newtheorem{de}[tw]{Definition}
\newtheorem{col}[tw]{Corollary}
\newtheorem{uwa}[tw]{Remark}
\begin{document}
\maketitle
\begin{abstract}We consider the quasi-static evolution of the thermo-plasticity model in which the evolution equation law for the inelastic strain is given by the Prandtl-Reuss flow rule. The thermal part of the Cauchy stress tensor is not linearised in the neighbourhood of a references temperature. This nonlinear thermal part imposed to add a damping term to the balance of the momentum, which can be interpreted as external forces acting on the material.

In general the dissipation term occurring in the heat equation is integrable function only and the standard methods can not be applied. Combining truncation techniques and Boccardo-Gallou\"et approach with monotone methods we prove an existence of renormalised solutions.
\end{abstract}

\section{Introduction and formulation of the problem}
The goal of the current paper is to study a perfect plasticity coupled with a heat conduction equation. We extend results obtained in \cite{ChelminskiOwczarekthermoI} and \cite{ChelminskiOwczarekthermoII}, where a constitutive equation for a plastic stain is of a Norton-Hoff type. Properties of the Norton-Hoff's flow rule allow to apply less complicated mathematical analysis. For example when the Minty-Browder's monotone trick was used we did not need an estimate of the time derivative of stress tensor. It is worth to mention, that in a classical approach to a purely mechanic Prandtl-Reuss model it turns out, that a plastic part is represented as a measure in a sense of Temam (see for example \cite{temam1}, \cite{Chelminski00} \cite{chel01}).

Firstly we present shortly the associated conservation laws from which the thermo-perfect plasticity system arises. All functions (apart from initial conditions) occurring in this article are functions dependent on the position $x\in\Omega$ and time $t\in(0,T)$, where the considered continuum occupies a bounded domain $\Omega\subset\R^3$ with boundary $\partial\Omega $ of class $C^2$ and has the constant density of mass $\rho>0$. $T>0$ is fixed length of the time interval.

Let us introduce the displacement vector $u$, the symmetric part of the gradient of the displacement $\ve(u)=\frac{1}{2}(\nabla_x u+\nabla_x^T u)$ (the linearised strain tensor), the total stress tensor $\sigma$ and the temperature of the body $\theta$.

In this article we deal with small deformations only, therefore the balance law of internal energy $e$ is in the following form
\begin{equation}
\label{balance}
\rho e_t + \dyw q=\sigma :\ve(u_t)+h\,,
\end{equation}
where $q$ is the heat flux, $h$ heat sources and $\ve(u_t)$ the symmetric part of the strain rate tensor. We assume that the constitutive equations for the total stress and the heat flux are in the form
\begin{equation}\begin{split}
\sigma&=\D(\ve(u)-\ve^p)-f(\theta)\id \,,\\[1ex]
q&=-\kappa\nabla\theta\,,
\label{consti}\end{split}
\end{equation}
where $\ve^p$ denotes the inelastic part of strain tensor,  $\D$ the given elasticity tensor (symmetric and positive definite on the space of symmetric $3\times 3$-matrices) and $\id$ the identity matrix. Entries of the tensor $\D$ do not depend on a material point $x$ and a time $t$. The first equation of (\ref{consti}) is a generalisation of Hook's law in which the thermal part of stress is given by the non-linear contituous function $f\colon\R\to \R$. The second one is Fourier's law with thermal conductivity $\kappa>0$. The standard density of an internal energy $e$ for this type of issues (quasistatic case) is in the form
\begin{equation}
\label{rozklad}
\rho e = c\theta+\frac{1}{2}\D^{-1}\T: \T\,,
\end{equation}
where $$\T=\sigma+f(\theta)\id=\D(\ve(u)-\epsp).$$ Substituting (\ref{consti}) and (\ref{rozklad}) to (\ref{balance})  we obtain the complete form of energy balance 
\begin{equation}
\label{enb}
c\theta_t-\kappa\Delta\theta+f(\theta)\dyw u_t=\T:\ve^p_t+h\,.
\end{equation}
To close our system of equations we need constitutive relation on inelastic part of the strain. In the theory of inelastic deformations for metals it is assumed that we know a evolution in time of $\ve^p$ {\it i.e.}
\begin{equation}
\label{fl}
\ve^p_t=(\in)\, G(T)\,,
\end{equation}
where $G$ is a function or, in some cases as well in the case considered in the current paper, a multivalued function. In the literature we can find a lot of examples of the function $G$ (see for instance \cite{alber,chegwia1,owcz2,ChelNeffOwczarek14} and many others). Observe that selection of the vector fields $G$ lead to study different models. In this article we consider the rate-independent Prandtl-Reuss model of the elasto-perfect plasticity. This constitutive law is studied most often in the literature \cite{Lionsfr}. For simplicity we study the Prandtl-Reuss model with von Mises criterion \cite{Mises1913,Lionsfr} {\it i.e.}
\begin{equation}
\label{P-R}
\ve^p_t\in \partial  I_{K}
(\T) \,,
\end{equation}
where the set of admissible elastic stresses $K$ is defined in the following form 
$$
K  = \{\T \in \S : |\PT| = |\T - \frac{1}{3} \tr (\T)
\id | \leq k \}
$$
and $k > 0$ is a given material constant (the yield limit). As it was mentioned previously $\id$ denotes the $3\times 3$ identity matrix, while $\S$ the set of symmetric $3\times 3$ real-valued matrices, hence  $P:\S\rightarrow P\S$ is the projector on the deviatoric part of symmetric matrices. The function $I_{K}$ is the indicator function of the admissible set $K$, this means that 
\begin{displaymath}
I_{K}(T)= \left\{ \begin{array}{ll}
0\quad & \textrm{for }\, T\in K\,,\\
\infty\quad & \textrm{for }\  T\notin K\,. \\
\end{array} \right.
\end{displaymath}
Finally, the function $\partial  I_{K}$ denotes the subgradient of the convex, proper, lower semicontinuous function $I_K$ in the sense of  convex analysis (see for instance \cite{AubFran}).

In this paper we consider the balance of forces in quasistatic case, that mean the inertial term $\rho u_{tt}$ is negligible.

Finally the main system studying in current article yields
\begin{equation}
\begin{split}
\dyw  \sigma(x,t)&=-F(x,t)-\dyw\D (\ve(u_t(x,t)))\,,\\[1ex]
\sigma(x,t)&=\D(\ve(u(x,t))-\ve^p(x,t))-f(\theta(x,t))\id \,,\\[1ex]
\varepsilon^p_t(x,t) &\in \partial  I_{K}
(\T(x,t))\,,\\[1ex]
\T(x,t)&=\D(\ve(u(x,t))-\ve^p(x,t))\,,\\[1ex]
\theta_t(x,t)-\Delta\theta(x,t)&=-f(\theta(x,t))\dyw u_t(x,t)+\ve^p_t(x,t): \T(x,t)\,,
\end{split}\label{eq:1.1}
\end{equation}
where 
\begin{equation*}
u\colon [0,T]\times\Omega\to \R^3\,,\qquad \ve^p\colon [0,T]\times\Omega\to \S\qquad\text{and}\qquad
\theta\colon [0,T]\times\Omega\to \R
\end{equation*}
are unknown functions. It is worth to emphasize, that the term $\dyw\D(\ve(u_t))$ on the right-hand side of the balance of forces, we treat as a regularisation not as a part of the stress tensor as it is treated in Kelvin-Voigt materials. We complete the system \eqref{eq:1.1} with the homogeneous Dirichlet boundary condition for the displacement and the Neumann boundary condition for the temperature
\begin{equation}
\label{eq:1.3}
\begin{split}
u(x,t)&=0\,,\\[1ex]
\frac{\partial\,\theta}{\partial\,n}(x,t)&=g_{\theta}(x,t)\,
\end{split}
\end{equation}
for $t\geq 0$ and $x\in \partial \Omega$, while the initial data are given by
\begin{equation}
\label{eq:1.4}
\begin{split}
u(x,0)&=u_0(x)\,,\\[1ex]
\ve^p(x,0)&=\ve^{p,0}(x)\,,\\[1ex]
\theta(x,0)&=\theta_0(x)\,
\end{split}
\end{equation}
for $x\in\Omega$.
Notice that in the system (\ref{eq:1.1}) the thermal stress, given by $-f(\theta)\id$, is not linearised in the neighbourhood of the reference temperature. The dissipation term on the right-hand side of energy balance $(\ref{eq:1.1})_5$ is usually in a space $L^1((0,T)\times\Omega;\R)$ only. From articles \cite{bocc89}, \cite{blanmur} and \cite{dall96} we deduce that a solution of heat equation with $L^1$-data is expected as a function from $L^p((0,T)\times\Omega;\R)$ for $p<\frac{N+2}{N}$, where $N$ denotes the dimension of a space. For this reason in this  paper it is assumed that the function $f:\R\rightarrow\R$ is continuous and satisfies the following growth condition
\begin{equation}\label{growth1}
|f(r)|\leq a+M|r|^{\alpha}\qquad \mathrm{for\,all}\qquad r\in\R\,,\,\, a,\,M\geq 0\,\,\, \mathrm{and}\,\,\, \alpha\in\Big(\frac{1}{2},\frac{5}{6}\Big)
\end{equation} 
 and there exists constant $C>0$ such that
\begin{equation}\label{growth2}
|f(r)|\leq C(1+|r|)^{\frac{1}{2}}\qquad \mathrm{for\,all}\qquad r\in\R_{-}\,.
\end{equation}
In general case $\alpha<\frac{N+2}{2N}$. The motivation of above assumptions is derived from the articles \cite{blangui} and \cite{BlanchardGuibe00} where a thermo-visco-elastisity ploblem for the Kelvin-Voigt material was studied. Nevertheless, it is worth to emphasize that in our paper we do not need assume that derivative of the function $f$ is bounded (as was assumed in \cite{blangui} and \cite{BlanchardGuibe00}).

Another point of view on thermo-mechnical problems is presented by S. Bartels and T. Roub{\'{\i}}{\v{c}}ek (see for example \cite{RouBart1} and \cite{RouBart2}) where authors use, so called, enthalpy transformation and consider energetic solutions. In both of those papers authors study Kelvin-Voigt viscous material, but in the article \cite{RouBart1} they consider a plasticity with hardening in quasistatic case, while in  \cite{RouBart2} the perfect plasticity in dynamical case is considered.

It is worth to emphasize the works \cite{GwiazdaKlaweSwierczewska014,GwiazdaKlaweSwierczewska14,tve-Orlicz} and \cite{KlaweOwczarek}, where the authors deal with similar type of thermo-visco-elasticity systems. In considered problems the thermal expansion does not appear, which means that the Cauchy stress tensor does not depend on temperature function. This main assumption leads to consider systems without additional damping term (the nonlinear term $f(\theta)\dyw u_t$ does not appear). Coupling between temperature and displacement occur only in flow rules. Using very special two level Galerkin approximation (proposed by Gwiazda at al.), the existence of weak solutions was proved. An important issue is the fact that the systems considered by Gwiazda at al. and in system (\ref{eq:1.1}), the total energy is conserved. Contrary to the systems analysed in \cite{JR,Haupt,cherack,bartczak,bartczak2} in which the lack of the total energy is observed. It is caused by the linearisation. The temperature occurring in nonlinear term of heat equation is linearised only (without any linearision of the Cauchy stress tensor).

Let us recall that we consider the system of equations in which the right-hand side of heat equation is expected in $L^1((0,T)\times\Omega;\R)$. It is known that in general, for integrable data a weak solution might not exist. Therefore DiPerna and Lions introduced a notion of renormalised solution for the Boltzmann equation in \cite{DiPernaLions}, to obtain well-posedness for this type problems. Such a notion was also adapted to elliptic equations with integrable data in \cite{BoccardoDiaz89,Murat93} and to parabolic equations in \cite{blan,blanmur,BlanchardMuratRedwane}

Now, we define a notion of renormalised solutions for the system (\ref{eq:1.1}).
Suppose that our data have the following regularity
\begin{equation}
\label{eq:Z0}
F\in L^{2}(0,T;L^2(\Omega;\R^3))\,,
\end{equation}
\begin{equation}
\label{eq:Z1}
g_{\theta}\in H^{1}(0,T;L^2(\partial\Omega;\R))\,,
\end{equation}
\begin{equation}
\label{eq:Z2}
u_0\in H^1(\Omega;\R^3) \textrm{,}\quad \ve^{p,0}\in L^2(\Omega;\SS)\textrm{,}\quad \theta_0\in H^1(\Omega;\R)\,.
\end{equation}
Now, we define a notion of renormalised solutions for the system (\ref{eq:1.1}). For any positive real number $K$, let us define the truncation function $\TC_K$ at height $K$ {\it i.e.} 
\begin{equation}\label{eq:ciecie}
\TC_{K}(r)=\min\{K,\max(r,-K)\}.
\end{equation}
Notice that $\TC_K(\cdot)$ is a real-valued Lipschitz function. Moreover, let us define $\varphi_K(r)=\int\limits _0^r \TC_K(s)\,\di s$, hence 
\begin{displaymath}
\varphi_K(r) = \left\{ \begin{array}{ll}
\frac{1}{2}r^2 & \textrm{if}\quad |r|\leq K\,,\\[1ex]
\frac{1}{2}K^2+K(|r|-K) & \textrm{if}\quad |r|>K\,,\\
\end{array} \right.
\end{displaymath}
and $\varphi_K$ is a $W^{2,\infty}(\R;\R)$-function with linear growth at infinity. Now we are ready to define a renormalised solutions to the problem (\ref{eq:1.1}).
\begin{de}
\label{de:1.1}
Suppose that the given data satisfy (\ref{eq:Z1}) and (\ref{eq:Z2}). A renormalised solution to the problem (\ref{eq:1.1}) - (\ref{eq:1.4}) is a vector $(u,\theta,\ve^p)$ satisfying the following conditions:\\[1ex]
{\bf 1.}
$$u\in H^{1}(0,T;H^1_0(\Omega;\R^3))\,,\quad \ve^p\in H^{1}(0,T;L^{2}(\Omega;\SS))\,,\\[1ex]$$
$$\theta\in C([0,T];L^1(\Omega;\R))\,,\quad
f(\theta)\in L^2(0,T;L^2(\Omega;\R))\,.$$
{\bf 2.}\hspace{2ex} 
$$\mathrm{div}\big(\sigma+\D (\ve(u_t))\big)\in L^{2}(0,T;L^2(\Omega;\R^3))$$
 and the equations $(\ref{eq:1.1})_1$-$(\ref{eq:1.1})_4$ are satisfied for almost all $(x,t)\in\Omega\times (0,T)$.\\[1ex]
{\bf 3.}\hspace{2ex} For each positive number $K>0$, $\TC_K(\theta)\in L^{2}(0,T;H^1(\Omega;\R))$ and the following equation
\begin{equation*}
\begin{split}
-\int\limits_0^{T}\int\limits_{\Omega}&S(\theta-\tilde{\theta})\,\varphi_t\,\di x\,\di t+\int\limits_{\Omega}S(\theta_0)\,\varphi(0,x)\,\di x\,\\[1ex]
+ \int\limits_0^{T}\int\limits_{\Omega}&S'(\theta-\tilde{\theta})\,\nabla(\theta-\tilde{\theta})\,\nabla\varphi\,\di x\,\di t+\int\limits_0^{T}\int\limits_{\Omega}S''(\theta-\tilde{\theta})\,|\nabla(\theta-\tilde{\theta})|^2\,\varphi\,\di x\,\di t\\[1ex]
&\qquad=\int\limits_0^{T}\int\limits_{\Omega} \Big(\epsp_t:\T - f(\theta)\mathrm{div}\,u_t\Big)\,S'(\theta-\tilde{\theta})\,\varphi\,\di x\,\di t
\end{split}
\end{equation*}
holds for all functions $\varphi\in C_0^{\infty}([0,T);H^1(\Omega;\R)\cap L^{\infty}(\Omega;\R))$ and $S\in C^{\infty}(\R;\R)$ such that $S'\in C^{\infty}_0(\R;\R)$, where $\tilde{\theta}$ is a solution of the following problem
\begin{eqnarray*}
\tilde{\theta}_t-\Delta\tilde{\theta}&=&0 \quad\,\, \textrm{in}\quad \Omega\times (0,T)\,,\nn\\[1ex]
\frac{\partial\,\tilde{\theta}}{\partial\,n}&=&g_{\theta}\quad \textrm{on}\quad \partial \Omega\times (0,T)\,,\nn\\[1ex]
\tilde{\theta}(x,0)&=&0\quad\,\, \textrm{in}\quad \Omega.\\
\end{eqnarray*}
{\bf 4.}\hspace{2ex} For any positive real number $C$ 
$$\TC_{K+C}(\theta)-\TC_K(\theta)\rightarrow 0\quad\mathrm{in}\quad L^2(0,T;H^1(\Omega;\R))$$
as $K$ goes to infinity.\\[1ex]
{\bf 5.}
$$u(x,0)=u_0(x),\quad \ve^p(x,0)=\ve^{p,0}(x),\quad \theta(x,0)=\theta_0(x)\,.$$
\end{de}
The main result of the present paper is the following theorem.
\begin{tw}$\mathrm{(Main\, Theorem)}$\\[1ex]
\label{tw:1.2}
Suppose that the boundary and initial data have the regularity required in (\ref{eq:Z1}) and (\ref{eq:Z2}), while the given function $F$ satisfies \eqref{eq:Z0}. Let the nonlinearity $f$ be a continuous function satisfying the growth conditions \eqref{growth1} and \eqref{growth2}. Moreover, let us assume that initial data satisfy
$$
|P\T(0,x)|=|P\D(\ve(u_0)-\epsp_0)|\leq k\,,\quad\text{for a.e.}\;x\in\Omega.
$$
 Then, for all $\T>0$ the system (\ref{eq:1.1}) with the boundary condition (\ref{eq:1.3}) and the initial condition (\ref{eq:1.4}) possesses a renormalised solution in the sense of Definition \ref{de:1.1}.
\end{tw}
The main idea of the proof of Theorem \ref{tw:1.2} is based on the Yosida approximation and a truncations. This means that we are going to use Yosida approximation to a maximal monotone multi-valued operator $\partial I_K$ and truncate the temperature function occurring in the Cauchy stress tensor. It impose that the dissipation term is also truncated. We would like to underline that the approximation is made on the same level as the truncution. In previous articles \cite{ChelminskiOwczarekthermoI} and \cite{ChelminskiOwczarekthermoII} two level approximation was used. Firstly, the considered system was approximate by a truncated systems (the first level). Next Yosida approximation was used, to prove an existence result for the obtained truncated systems (the second level). An another difference is the fact that in \cite{ChelminskiOwczarekthermoII} was applied Minti's monotonicity trick, to characterise a weak limits of nonlinearities and an estimate for a time derivative of the stress was not required. Here the flow rule does not have a structure of power law and Minti's trick is unhelpful. Therefore an additional estimate on a time derivative of the stress tensor is needed.  

Let us shortly summarize the contents of current article. 
In the second section we introduce Yosida approximation together with a truncation of the model.
Next we prove an existence of solution of the approximated model. The third section is focused on estimation for the approximated model, which are needed to pass to the limit with approximation. In the fourth section we pass to the limit and prove the main theorem (Theorem \ref{tw:1.2}).
\subsection{Transformation to a homogeneous boundary-value problem with respect to the temperature}
To make the consecutive calculation easier we transform the considered problem to the homogeneous boundary one. To proceed this, let us firstly consider the following boundary-initial linear parabolic problem:
\begin{equation}\label{eq:2.4}\left\{
\begin{array}{rl}
\tilde{\theta}_t(x,t)&-\Delta\tilde{\theta}(x,t)=0,\\[1ex]
\frac{\partial\,\tilde{\theta}}{\partial\,n}(x,t)&=g_{\theta}(x,t)\quad \textrm{ for}\quad x\in \partial \Omega \quad\textrm{and}\quad  t\geq 0\,,\\[1ex]
\tilde{\theta}(x,0)&=0\quad\,\,\,\, \textrm{ for}\quad x\in \Omega.
\end{array}\right.
\end{equation}

Assuming that $g_{\theta}$ satisfies \eqref{eq:Z1} we conclude that the system (\ref{eq:2.4}) possesses a solution $\tilde\theta\in L^{\infty}(0,T;H^1(\Omega;\R))$ such that $\tilde{\theta}_t\in L^2(0,T;L^2(\Omega;\R))$. Additionally the following estimate 
\begin{equation}
\label{eq:2.6}
\|\tilde{\theta}_t\|_{L^{2}(0,T;L^2(\Omega))}+\|\tilde{\theta}\|_{L^{\infty}(0,T;H^1(\Omega))}\leq D \,\|g_{\theta}\|_{H^{1}(0,T;L^2(\partial\Omega))} 
\end{equation} holds.
Now, if we denote by $(u,\epsp,\hat{\theta})$ the solution to problem (\ref{eq:1.1}-\ref{eq:1.4}) and define $\theta:=\hat{\theta}-\tht$, we observe that we can write the investigated problem equivalently {\it i.e.} for $x\in\Omega$ and $t\in[0,T]$
\begin{equation}\label{eq:HB}
\begin{split}
\dyw  \sigma(x,t)&=-F(x,t)-\dyw\D (\ve(u_t(x,t)))\,,\\[1ex]
\sigma(x,t)&=\D(\ve(u(x,t))-\ve^p(x,t))-f(\hat{\theta}(x,t)-\tht(x,t))\id \,,\\[1ex]
\varepsilon^p_t(x,t) &\in \partial  I_{K}
(\T(x,t))\,,\\[1ex]
\T(x,t)&=\D(\ve(u(x,t))-\ve^p(x,t))\,,\\[1ex]
\theta_t(x,t)-\Delta\theta(x,t)&=-f(\hat{\theta}(x,t)-\tht(x,t))\dyw u_t(x,t)+\ve^p_t(x,t): \T(x,t)\,,
\end{split}
\end{equation}
with the initial-boundary conditions in the following form
\begin{equation}
\label{eq:Ini-Bond}
\begin{array}{rclrcl}
u_{|_{\partial\Omega}}&=&0,\quad&
\frac{\partial\,\theta}{\partial\,n}_{|_{\partial\Omega}}&=&0, \\[1ex]
\theta(0)=\hat{\theta}_0&=&\theta_{0},&
u(0)&=&u_0,\quad
\ve^p(0)=\ve^{p,0}.
\end{array}
\end{equation}
\section{Yosida approximation and truncation of the problem}
\setcounter{equation}{0}
The first step to handle with our problem is to investigate Yosida approximation for the subgradient and, simultaneously, the truncations of terms $\theta+\tht$ and $\T:\epsp_t$. Let us emphasize that the same parameter $\lambda$ is used in truncation and Yosida approximation and this make a difference comparing with method used in \comment{cytowanie}.

\subsection{Definition of Yosida approximation and existence theorem}
To produce an approach to problem \eqref{eq:HB}, \eqref{eq:Ini-Bond} we apply to the multifunction $\partial I_K$ Yosida approximation $Y_\lambda$ (for details see {\it e.g.} \cite{AubFran}).  Similarly as in \cite{cherack} we obtain for each $\lambda>0$
\begin{equation}
Y_\lambda (\T) = 
\frac{(|\PT| - k)_+}{2\lambda} \frac{\PT}{|\PT|},
\label{eq:Yosida}
\end{equation}
where $(\xi)_+:=\max\{0,\xi\}$. For $\PT=0$ we set the value $0$ instead of $\PT/|\PT|$. Thus for all $\lambda>0$ we have to consider the following system
\begin{equation}
\label{eq:Appr}
\begin{split}
\dyw  \,\sigma^\lambda&=-F-\dyw \D (\ve(u_t))\,,\\
\sigma^\lambda&=\D(\ve(\ul)-\epspl)-f\big(\TC_{\frac{1}{\lambda}}(\thl+\tilde{\theta})\big)\id\,,\\[1ex]
\epspl_t&= Y_\lambda (\T)\,,\\[1ex]
\Tl&=\D(\ve(\ul)-\epspl)\,,\\[1ex]
\thl_t-\Delta\thl&=-f\big( \TC_{\frac{1}{\lambda}}(\thl+\tilde{\theta})\big)\dyw  \ul_t+\TC_{\frac{1}{\lambda}}\Big(\epspl_t: \Tl\Big)\,,
\end{split}
\end{equation}
where the function $\TC_{\frac{1}{\lambda}}(\cdot)$ is the truncation at height $1/\lambda>0$ defined analogously as in \eqref{eq:ciecie}.
The system  is considered with the same boundary and initial conditions as the system \eqref{eq:HB} ({\it i.e.} \eqref{eq:Ini-Bond}).

The main theorem of this section yields:
\begin{tw}[Existence of the solution to approximated system] For each $\lambda>0$ the system (\ref{eq:Appr}) with initial and boundary data (\ref{eq:Ini-Bond}) has a $L^2$-solution $(\ul,\Tl,\epspl,\thl)\in L^\infty(0,T;H^1(\Omega;\R^3)\times (L^2(\Omega;\S))^2\times H^1(\Omega;\R)$. Moreover it holds that $(\ul_t,\Tl_t,\epspl_t,\thl_t)\in L^\infty(0,T;L^2(\Omega;\R^3)\times (L^2(\Omega;\S))^2\times L^2(\Omega;\R)$.
\label{tw:istn-lambda}
\end{tw}
To prove Theorem \ref{tw:istn-lambda} we use fixed point theorems.
\subsection{Proof of Theorem \ref{tw:istn-lambda}}
Since we fix $\lambda>0$ we omit the index $\lambda$ when we denote unknown functions. We shall construct the compact operator 
$$
\pis{A}\colon L^r((0,T)\times\Omega;\R)\rightarrow L^r((0,T)\times\Omega;\R) 
$$ 
for some fixed $r\in(1,2).$ For this purpose we fix a function $\ths\in L^r((0,T)\times\Omega;R)$ and we consider the first auxiliary problem on $[0,T]\times\Omega$:
\begin{equation}\label{AP1}
\begin{split}
\dyw  \sigma=&-F-\dyw \D (\ve(u_t))\,,\\
\sigma=&\D(\ve(u)-\epsp)-f\big(\TC_{\frac{1}{\lambda}}(\ths+\tilde{\theta})\big)\id\,,\\[1ex]
\epsp_t=& Y_\lambda (\T)\,,\\[1ex]
\T=&\D(\ve(u)-\epsp)\,,\\[1ex]
u|_{\partial\Omega}=&0\,,\qquad
u|_{t=0}=u_0\,,\qquad
\epsp|_{t=0}=\ve^{p,0}\,.
\end{split}
\end{equation}
\begin{lem}Assume  $\ths\in L^r((0,T)\times\Omega;\R)$. Then there exists the unique solution $(\sigma, \epsp, u)$ to \eqref{AP1} satysfying $\sigma\in L^2((0,T)\times\Omega;\S))$ and  $\epsp\in H^1(0,T;L^{2}(\Omega;\S))$ while $u\in H^1(0,T;\HoR)$. Additionally the following estimate holds
\begin{equation}
\begin{split}
\norm{\sigma&}_{L^\infty(0,T;L^ {2}(\Omega))}+ \norm{\epsp}_{H^1(0,T;L^ {2}(\Omega))}+ \norm{u}_{H^1(0,T;\Ho)}\\[1ex]
\leq & E(T)\left(\norm{f\big(\TC_{\frac{1}{\lambda}}(\ths+\tilde{\theta}))}_{L^2((0,T)\times\Omega)} +\norm{\ve^{p,0}}_{L^2(\Omega)}+ \norm{u_0}_{\Ho} +\norm{F}_{L^2((0,T)\times\Omega)}\right).
\end{split}
\label{eq:ap1a}
\end{equation}
\label{lem:ap1}
\end{lem}
\begin{proof} We apply the Banach fixed point theorem. Let $\epsps\in L^2(0,T;L^2(\Omega;\S))$ and using facts proven in Appendix we solve the linear elasticity problem in the form:
\begin{equation}
\begin{split}    
-\dyw\D(\ve(w))-\dyw\D(\ve(w_t))&= -\dyw\D(\epsps)-\dyw f\big(\TC_{\frac{1}{\lambda}}(\ths+\tilde{\theta})\big)\id+F\,,\\[1ex]
w|_{\partial\Omega}=0\,,\quad&\quad w|_{t=0}=u_0\,.
\end{split}
\label{eq:elast}
\end{equation}
Using Lemma \ref{lem:LE} we obtain that there exists a unique solution $w\in H^1(0,T;\HoR)$. Additionally $w$ satisfies the following estimate
\begin{equation*}
\begin{split}
\norm{w}_{H^1(0,T;\Ho)}\leq& CT\left(e^{CT}+1\right)\Big(\norm{\epsps}_{L^2((0,T)\times\Omega)}+ \norm{f\big(\TC_{\frac{1}{\lambda}}(\ths+\tilde{\theta})\big)}_{L^2((0,T)\times\Omega)}\\[1ex]
&\phantom{ CT\left(e^{CT}+1\right)\Big(} +\norm{u_0}_{\Ho}
+ \norm{F}_{L^2((0,T)\times\Omega)}\Big),\\
%
%
\leq& CT\left(e^{CT}+1\right)\Big(\norm{\epsps}_{L^2((0,T)\times\Omega)}+ 1+\norm{u_0}_{\Ho}
+ \norm{F}_{L^2((0,T)\times\Omega)}\Big),
\end{split}
\end{equation*}
where the constant $C>0$ depends only on the operator $\D$, the constant $\lambda$ and the shape of the set $\Omega$. We use the fact $f$ is continuous, thus $\norm{f\big(\TC_{\frac{1}{\lambda}}(\ths+\tilde{\theta})\big)}_{L^\infty(\Omega)}\leq\max_{|\xi|\leq1/\lambda}|f(\xi)|$. We denote $D(T):=CT\left(e^{CT}+1\right)$ and put 
$$
\sigma=\D(\ve(w)-\epsps)+\D(\ve(w_t))-f\big(\TC_{\frac{1}{\lambda}}(\ths+\tilde{\theta})\big)\id.
$$
Therefore
\begin{equation*}
\begin{split}
\norm{\sigma}_{L^2((0,T)\times\Omega)}\leq& C\left(\norm{w}_{H^1(0,T;\Ho)}+ \norm{\epsps}_{L^2((0,T)\times\Omega)}+\norm{f\big(\TC_{\frac{1}{\lambda}}(\ths+\tilde{\theta})\big)}_{L^2((0,T)\times\Omega)}\right)\\
\leq& D(T)\Big(\norm{\epsps}_{L^2((0,T)\times\Omega)}+ 1 +\norm{u_0}_{\Ho}+ \norm{F}_{L^2((0,T)\times\Omega)}\Big).
\end{split}
\end{equation*}
Now we define the following operator 
$$
\epsp(x,t)={\cal R}(\epsps)(x,t):=\int\limits _0^t Y_\lambda(\D(\eps(u(x,\tau))-\epsp(x,\tau))\,\di \tau   + \epsp_0(x).
$$ 
Using $Y_\lambda(0)=0$ and Lipschitz-continuity of $Y_\lambda$, we obtain the following estimate
\begin{equation*}
\begin{split}
|\epsp(x,t)|\leq&\int\limits _0^t\left|Y_\lambda(\D(\eps(u(x,\tau))-\epsp(x,\tau))\right|\di \tau     +|\ve^p_0(x)|\\[1ex]
\leq&\frac{1}{2\lambda}\int\limits _0^t\left|\D(\eps(u(x,\tau))-\epsp(x,\tau)\right|\di \tau    +  
|\ve^p_0(x)|.
\end{split}
\end{equation*}
The norm of $\epsp$ is estimated in the following way
$$
\norm{\epsp}_{L^2((0,T)\times\Omega)}\leq \frac{CT}{\lambda}\norm{\D(\eps(u)-\epsp}_{L^2((0,T)\times\Omega)}+ \norm{\ve^p_0}_{L^2(\Omega)}.
$$
Observe that the operator ${\cal R}$ is proper defined. We choose a sufficiently short time interval to have that the operator $\pis{R}$ be a contraction. Indeed, let us fix $\epsps_1,\epsps_2\in L^2((0,T)\times\Omega);\S)$. By solving \eqref{eq:elast} we obtain $w_1, w_2$ respectively. Next, using reasoning similar to the previous one, we have
\begin{align*}
\epsp_1(x,t)&=\int\limits _0^t Y_\lambda(\T_1(x,\tau))\di \tau    +\ve^p_0(x),\\
\epsp_2(x,t)&=\int\limits _0^t Y_\lambda(\T_2(x,\tau))\di \tau    +\ve^p_0(x),
\end{align*}
and $\T_i=\D(\ve(w_i)-\epsps_i)$ for $i=1,2$.
We subtract $\epsp_1-\epsp_2$ to obtain
\begin{equation*}
\begin{split}\norm{\epsp_1(t)-\epsp_2(t)}_{L^2(\Omega)}&\leq\int\limits _0^t\norm{Y_\lambda(\T_1(\tau))-Y_\lambda(\T_2(\tau))}_{L^2(\Omega)}\di \tau   \\
\end{split}
\end{equation*}
hence
$$
\norm{\epsp_1-\epsp_2}_{L^2((0,T)\times\Omega)}\leq \frac{T}{2\lambda}\norm{\T_1-\T_2}_{L^2((0,T)\times\Omega)}.
$$
As a consequence of the linearity of the problem \eqref{eq:elast} we obtain that a difference $w_1-w_2$ satisfies
\begin{equation*}
\begin{split}    
-\dyw\D(\ve(w_1-w_2))-\dyw\D(\ve((w_1-w_2)_t))&= -\dyw\D(\epsp_1-\epsp_2)\,,\\[1ex]
(w_1-w_2)|_{\partial\Omega}&=0\,,\\[1ex]
(w_1-w_2)|_{t=0}&=0\,.
\end{split}
\end{equation*}
The following estimate holds
$$
\norm{w_1-w_2}_{H^1(0,T; \Ho )}\leq D(T)\norm{\epsps_1-\epsps_2}_{L^2((0,T)\times\Omega)},
$$
where the constant $D(T)$ is the same as previously. Hence
\begin{align*}
\norm{\T_1-\T_2}_{L^2((0,T)\times\Omega)}
&\leq C\left(\norm{w_1-w_2}_{L^2(0,T; \Ho )}+ \norm{\epsps_1-\epsps_2}_{L^2((0,T)\times\Omega)}\right)\\
&\leq CD(T)\norm{\epsps_1-\epsps_2}_{L^2((0,T)\times\Omega)}.
\end{align*}
Now we estimate the difference $\pis{R}(\epsps_1)-\pis{R}(\epsps_2)$ as follows
\begin{equation}
\norm{\pis{R}(\epsps_1)-\pis{R}(\epsps_2)}_{L^2((0,T)\times\Omega)}= \norm{\epsp_1-\epsp_2}_{L^2((0,T)\times\Omega)} \leq CTD(T)\norm{\epsps_1-\epsps_2}_{L^2((0,T)\times\Omega)},
\label{eq:contr1}
\end{equation}
where the constant $D(T)$ is the same as previously and does not depend on initial data.  Therefore we are able to choose a sufficiently small $T_1>0$ to get that $\pis{R}_1$ is a contraction on $L^2((0,T_1)\times\Omega;\S)$. Thus using the Banach fixed point theorem we obtain existence of $\epsp\in L^2((0,T_1)\times\Omega;\S)$ such that $\pis{R}_1(\epsp)=\epsp$. Hence the problem (AP1) possesses the unique solution $(\sigma, \epsp,u)\in L^2((0,T_1)\times\Omega;\S)\times L^2((0,T_1)\times\Omega;\S)\times H^1(0,T_1;\HoR)$. We repeat the argumentation analogous to the proof of Lemma \ref{lem:LE} to extend obtained solution to the whole interval $(0,T)$ since the estimate \eqref{eq:contr1} is obtained independently of the initial data.

Finally we prove the estimate \eqref{eq:ap1a}. Thus using the second equation in \eqref{AP1} we obtain the following inequalities
\begin{align*}
\norm{\sigma(t)}_{L^2(\Omega)}
& \leq  C\left(\norm{u(t)}_{\Ho}+ \norm{\epsp(t)}_{L^2(\Omega)}+\norm{f\big(\TC_{\frac{1}{\lambda}}(\ths(t)+\tilde{\theta}(t))\big)}_{L^2(\Omega)}\right)\\[1ex]
& \leq  C(T)\left(\norm{\epsp(t)}_{L^2(\Omega)}+ \norm{f\big(\TC_{\frac{1}{\lambda}}(\ths(t)+\tilde{\theta}(t))\big)}_{L^2(\Omega)} + \norm{u_0}_{\Ho}\right.\\[1ex]
&\left.\phantom{\leq  C(T)\Big(}\qquad\qquad+  \norm{F(t)}_{L^2(\Omega)}\right)\\[1ex]
&\leq  C(T)\left(\frac{1}{\lambda}\int\limits _0^t\norm{\T(\tau)}_{L^2(\Omega)}\di\tau+\norm{f\big(\TC_{\frac{1}{\lambda}}(\ths(t)+\tilde{\theta}(t))\big)}_{L^2(\Omega)}+ \norm{u_0}_{\Ho}\right.\\[1ex]
&\left.\phantom{\leq C(T)\Big(\frac{1}{\lambda}\int\limits _0^t\norm{\T(\tau)}_{L^2(\Omega)}+1}+ \norm{\epsp_0}_{L^2(\Omega)} + \norm{F(t)}_{L^2(\Omega)}\right).\nonumber\\[1ex]
&\leq  C(T)\left(\frac{1}{\lambda}\int\limits _0^t\norm{\sigma(\tau)}_{L^2(\Omega)}+\norm{f\big(\TC_{\frac{1}{\lambda}}(\ths(\tau)+\tilde{\theta}(\tau))\big)}_{L^2(\Omega)}\,\di\tau\right.\\[1ex]
&\phantom{\leq C(T)\Big(\frac{1}{\lambda}}+\norm{f\big(\TC_{\frac{1}{\lambda}}(\ths(t)+\tilde{\theta}(t))\big)}_{L^2(\Omega)}+ \norm{u_0}_{\Ho}\\[1ex]
&\left.\phantom{\leq C(T)\Big(\frac{1}{\lambda}\int\limits _0^t\norm{\sigma(\tau)}_{L^2(\Omega)}+1}+ \norm{\epsp_0}_{L^2(\Omega)} + \norm{F(t)}_{L^2(\Omega)}\right).\nonumber
\end{align*}
We apply the Gronwall inequality to the expression above and obtain
\begin{equation}\label{eq:sig1}
\begin{split}
\norm{\sigma(t)}_{L^2(\Omega)} \leq& C(T)\left(\norm{\epsp_0}_{L^2((\Omega)}+ \norm{f\big(\TC_{\frac{1}{\lambda}}(\ths+\tilde{\theta})\big)}_{L^2((0,T)\times\Omega)}\right.\\[1ex]
&\phantom{C(T)\Bigg(\norm{\epsp_0}}\left.+ \norm{u_0}_{\Ho}+ \norm{F}_{L^2((0,T)\times\Omega)}\right)
\end{split}
\end{equation}
and we deduce the required estimate for $\norm{\sigma}_{L^2((0,T)\times\Omega)}$ is completed. Moreover we observe that 
\begin{equation}\label{eq:epsp1}
\begin{split}
\norm{\epsp_t}_{L^2((0,T)\times\Omega)}\leq& \frac{C(T)}{\lambda}\norm{\sigma}_{L^2((0,T)\times\Omega)}\\[1ex]
\leq&\frac{C(T)}{\lambda}\left(\norm{\epsp_0}_{L^2((\Omega)}+ \norm{f\big(\TC_{\frac{1}{\lambda}}(\ths+\tilde{\theta})\big)}_{L^2((0,T)\times\Omega)}\right.\\[1ex]
&\phantom{C(T)\Bigg(}\left.+ \norm{u_0}_{\Ho}+ \norm{F}_{L^2((0,T)\times\Omega)}\right)
\end{split}
\end{equation}
and we get an estimate of a norm $\norm{\epsp_t}_{L^2((0,T)\times\Omega)}$. To estimate a term $\norm{u}_{H^1(0,T;\Ho)}$ we once again use the Lemma \ref{lem:LE} as in problem \eqref{eq:elast}. Therefore
\begin{equation}
\begin{split}
\norm{u}_{H^1(0,T;\Ho)}\leq& D(T)\Big(\norm{\epsp}_{L^2((0,T)\times\Omega)}+  \norm{f\big(\TC_{\frac{1}{\lambda}}(\ths+\tilde{\theta}))}_{L^2((0,T)\times\Omega)}\\[1ex]
&\phantom{D(T)\Big(}+ \norm{F}_{L^2((0,T)\times\Omega)}+ \norm{u_0}_{\Ho}\Big).
\end{split}\label{eq:d tu}
\end{equation}
The inequalities \eqref{eq:sig1} together with \eqref{eq:epsp1} and \eqref{eq:d tu} give us \eqref{eq:ap1a}.
\end{proof}
%
%
Next, we solve the second auxiliary problem on $(0,T)\times\Omega$ {\it i.e.}
\begin{equation}\label{AP2}
\begin{split}
\theta_t-\Delta\theta=&-f\left(\TC_{\frac{1}{\lambda}}(\ths+\tilde{\theta})\right)\dyw u_t+\TC_{\frac{1}{\lambda}}\big(\epsp_t:\T\big),\\[1ex]
\frac{\partial\,\theta}{\partial\,n}{|_{\partial\Omega}}=&0, \qquad \theta{|_{t=0}}=\theta_{0},
\end{split}
\end{equation}
where $\ths\in L^r((0,T)\times\Omega;\R)$ was fixed before, while $(u,\T, \epsp)$ are solution of the problem~\eqref{AP1}.
\begin{lem} Let $\theta_0\in W^{2-2/r,r}(\Omega;\R)$, $RHS\in L^r((0,T)\times\Omega;\R)$, where
$$
RHS=-f\left(\TC_{\frac{1}{\lambda}}(\ths+\tilde{\theta})\right)\dyw u_t+\TC_{\frac{1}{\lambda}}\big(\epsp_t:\T\big).
$$
Then there exists a solution of problem \eqref{AP2} $\theta$, unique in the class
$$
\theta\in L^r(0,T;W^{2,r}(\Omega;\R)),\quad\theta_t\in L^r((0,T)\times\Omega;\R),\quad\theta\in C([0,T];W^{2-2/r,r}(\Omega;\R)).
$$
Moreover, there exists constant $C>0$ independent of the given data such that
\begin{equation}\label{eq:ap2a}
\begin{split}
\norm{\theta}_{L^\infty((0,T);W^{2-2/r,r}(\Omega))}+ \norm{\theta_t}_{L^r((0,T)\times\Omega)}+ \norm{\Delta\theta}_{L^r((0,T)\times\Omega)}\\
\leq C\left(\norm{RHS}_{L^r((0,T)\times\Omega)}+\norm{\theta_0}_{W^{2,r}(\Omega)}\right).
\end{split}
\end{equation}
\label{lem:AP2}
\end{lem}
\begin{uwa} The Lemma \ref{lem:AP2} is the standard result from theory of the parabolic equations. See for example \cite{Amann93,Amann95}.
\end{uwa}
\begin{uwa}
According to assumption given in current paper we have for a bounded domain $\Omega$ with the boundary of the class $C^2$:
\begin{equation*}
\begin{split}
RHS\in& L^2((0,T)\times\Omega;\R)\subset L^r((0,T)\times\Omega;\R),\\
\theta_0\in& H^1(\Omega;\R)\subset W^{2-2/r,r}(\Omega;\R)
\end{split}
\end{equation*}
for any $r\in(1,2)$. 
\end{uwa}
The considerations above allows us to define operator 
$$
\pis{A}\colon L^r((0,T)\times\Omega;\R)\to L^r((0,T)\times\Omega;\R)
$$
{\it i.e.} first we fix $\ths\in L^r((0,T)\times\Omega;\R)$, then we solve the problem \eqref{AP1} and next we solve the problem \eqref{AP2} using as a given data the solution of \eqref{AP1}. Finally we put $\pis{A}(\ths):=\theta$, where $\theta\in L^r((0,T)\times\Omega;\R)$ is the solution of \eqref{AP2} obtained through the procedure described above.
%
%
\begin{lem} The operator  
$$
\pis{A}\colon L^r((0,T)\times\Omega;\R)\to L^r(0,T;W^{2,r}(\Omega;\R)\cap W^{1,r}(0,T;L^r(\Omega);\R)
$$
is continuous for any $r\in(1,2)$.
\label{lem:A-cont}
\end{lem}
\begin{proof}We proceed analogously to the proof of Proposition 1 \cite{BartOwcz1}. We fix $\ths_1,\ths_2\in L^r((0,T)\times\Omega)$. Our aim is to examine a norm of the difference $\pis{A}(\ths_1)-\pis{A}(\ths_2)$. First we consider the problem \eqref{AP1} for $\ths_1$ and $\ths_2$. Solutions of this problem ($(u_1,\T_1,\epsp_1)$ and $(u_2,\T_2,\epsp_2)$ respectively) satisfy the following system:
\begin{equation}\label{roznica}
\begin{split}
\dyw \T\r+ \dyw\D(\ve(u_{\Delta,t})=& \dyw \left(\left(\TC_{\frac{1}{\lambda}}(\ths_1+\tht)- \TC_{\frac{1}{\lambda}}(\ths_2+\tht)\right)\id\right)\\
\epsp_{\Delta,t}=& Y_{\lambda}(\T_1)-Y_{\lambda}(\T_2)\,,\\
\T\r=& \D(\ve(u\r)-\epsp\r)\,,\\
u\r|_{\partial\Omega}=0,\qquad& u\r|_{t=0}=0,\qquad \epsp\r|_{t=0}=0\,,
\end{split}
\end{equation}
where we use a notation $a_\Delta$ for a difference {\it i.e.} $a_\Delta:=a_1-a_2$. Using similar investigation as in the proof of the Lemma \ref{lem:ap1} we estimate as follows
\begin{equation*}
\norm{\epsp\r}_{L^2((0,T)\times\Omega)}\leqslant \left(\int\limits _0^T \norm{Y_{\lambda}(\T_1(\tau))-Y_{\lambda}(\T_2(\tau))}_{L^2(\Omega)}^2\,\di \tau\right)^{1/2}    \leqslant \frac{T}{2\lambda}\norm{\T\r}_{L^2((0,T)\times\Omega)}.
\end{equation*}
The estimate \eqref{eq:LEestim} from Lemma \ref{lem:LE} yields
\begin{equation}\label{eq:u_r}
\begin{split}
\norm{u\r}&_{H^1(0,T;\Ho)} \\
\leq& C(T)\Big(\norm{f\left(\TC_{\frac{1}{\lambda}}(\ths_1+\tht)\right)-f\left( \TC_{\frac{1}{\lambda}}(\ths_2+\tht)\right)}_{L^2((0,T)\times\Omega)}+\norm{\epsp\r}_{L^2((0,T)\times\Omega)}\Big)\\[1ex]
\leq& C(T)\Big(\norm{f\left(\TC_{\frac{1}{\lambda}}(\ths_1+\tht)\right)-f\left( \TC_{\frac{1}{\lambda}}(\ths_2+\tht)\right)}_{L^2((0,T)\times\Omega)}+\frac{1}{\lambda}\norm{\T\r}_{L^2((0,T)\times\Omega)}\Big).
\end{split}
\end{equation}
Similar argumentation as in the proof of Lemma \ref{lem:ap1} leads us to
\begin{equation}\label{eq:eps+T}
\begin{split}
\norm{\T\r}_{L^2((0,T)\times\Omega)}+ \norm{\epsp\r}&_{H^1(0,T;L^2(\Omega))}\\[1ex]
&\leq C(T)\norm{f\left(\TC_{\frac{1}{\lambda}}(\ths_1+\tht)\right)-f\left( \TC_{\frac{1}{\lambda}}(\ths_2+\tht)\right)}_{L^2((0,T)\times\Omega)}.
\end{split}
\end{equation}
The inequality \eqref{eq:u_r} together with \eqref{eq:eps+T} imply
\begin{equation*}
\begin{split}
\norm{u\r}_{H^1(0,T;\Ho)}+\norm{\T\r}&_{L^2((0,T)\times\Omega)}+ \norm{\epsp\r}_{H^1(0,T;L^2(\Omega))}\\[1ex]
&\leq D(T)\norm{f\left(\TC_{\frac{1}{\lambda}}(\ths_1+\tht)\right)-f\left( \TC_{\frac{1}{\lambda}}(\ths_2+\tht)\right)}_{L^2((0,T)\times\Omega)}.
\end{split}
\end{equation*}
Notice that $f\left(\TC_{\frac{1}{\lambda}}((\cdot)-\tht)\right)$ defines a continuous operator from $L^r((0,T)\times\Omega;\R)$ to $L^2((0,T)\times\Omega;\R)$ (defined as $\theta\mapsto f\left(\TC_{\frac{1}{\lambda}}(\theta-\tht)\right)$, compare with \cite[Theorem 7.19]{Dudleybook}), for any $\epsilon_1>0$ there exists $\delta_1>0$ such that if $\norm{\ths\r}_{L^r((0,T)\times\Omega;\R)}\leq\delta_1$ then
\begin{equation}\label{eq:eps_1}
\norm{u\r}_{H^1(0,T;\Ho)}+\norm{\T\r}_{L^2((0,T)\times\Omega)}+ \norm{\epsp\r}_{H^1(0,T;L^2(\Omega))}
\leq \epsilon_1.
\end{equation}
In the next step we consider the second auxiliary problem \eqref{AP2} with $(u_i,\T_i,\epsp_i)$ being a solution of the first auxiliary problem \eqref{AP2} for $\ths_i$ respectively for $i=1,2$. Therefore a difference of solutions $\theta\r:=\theta_1-\theta_2$ satisfies the following system:
\begin{equation}
\begin{split}
\theta_{\Delta,t}-\Delta\theta\r=&-f\left(\TC_{\frac{1}{\lambda}}(\ths_1+\tilde{\theta})\right)\dyw u_{1,t}+ f\left(\TC_{\frac{1}{\lambda}}(\ths_2+\tilde{\theta})\right)\dyw u_{2,t}\\[1ex]
&+\TC_{\frac{1}{\lambda}}\big(\epsp_{1,t}:\T_1\big)- \TC_{\frac{1}{\lambda}}\big(\epsp_{2,t}:\T_2\big),\\[1ex]
\frac{\partial\,\theta}{\partial\,n}{|_{\partial\Omega}}=0,& \qquad \theta{|_{t=0}}=0.
\end{split}
\end{equation}
The same type of estimates as in Lemma \ref{lem:AP2} give us:
\begin{equation}\label{eq:th_r1}
\begin{split}
\norm{\theta&\r}_{W^{1,r}(0,T;L^r(\Omega))}+ \norm{\theta\r}_{L^r(0,T;W^{2,r}(\Omega))}\\[1ex]
%
%
&\leq C(T)\Big(\norm{-f\left(\TC_{\frac{1}{\lambda}}(\ths_1+\tilde{\theta})\right)\dyw u_{1,t}+ f\left(\TC_{\frac{1}{\lambda}}(\ths_2+\tilde{\theta})\right)\dyw u_{2,t}}_{L^r((0,T)\times\Omega)}\\[1ex]
&\phantom{\leq C(T)\Big( MMMMMMM} + \norm{\TC_{\frac{1}{\lambda}}\big(\epsp_{1,t}:\T_1\big)- \TC_{\frac{1}{\lambda}}\big(\epsp_{2,t}:\T_2\big)}_{L^r((0,T)\times\Omega)} \Big).
\end{split}
\end{equation}
Let us consider the first therm on the right-hand side of \eqref{eq:th_r1}:
\begin{equation}
\begin{split}
\norm{-&f\left(\TC_{\frac{1}{\lambda}}(\ths_1+\tilde{\theta})\right)\dyw u_{1,t}+ f\left(\TC_{\frac{1}{\lambda}}(\ths_2+\tilde{\theta})\right)\dyw u_{2,t}}_{L^r((0,T)\times\Omega)}\\[1ex]
%
%
\leq& \norm{f\left(\TC_{\frac{1}{\lambda}}(\ths_1+\tilde{\theta})\right)\dyw u_{\Delta,t}}_{L^r((0,T)\times\Omega)}\\[1ex]
& + \norm{\left(f\left(\TC_{\frac{1}{\lambda}}(\ths_1+\tilde{\theta})\right)- f\left(\TC_{\frac{1}{\lambda}}(\ths_2+\tilde{\theta})\right)\right)\dyw u_{2,t}}_{L^r((0,T)\times\Omega)}\\[1ex]
%
%
\leq& \norm{f\left(\TC_{\frac{1}{\lambda}}(\ths_1+\tilde{\theta})\right)}_{L^{2r/(2-r)}((0,T)\times\Omega)}\norm{\dyw u_{\Delta,t}}_{L^2((0,T)\times\Omega)}\\[1ex]
& + \norm{\left(f\left(\TC_{\frac{1}{\lambda}}(\ths_1+\tilde{\theta})\right)- f\left(\TC_{\frac{1}{\lambda}}(\ths_2+\tilde{\theta})\right)\right)}_{L^{2r/(2-r)}((0,T)\times\Omega)}\norm{\dyw u_{2,t}}_{L^2((0,T)\times\Omega)}.
\end{split}\label{eq:f1-f2}
\end{equation}
Using the estimate \eqref{eq:ap1a} for $u_2$ and continuity of the operator $f\left(\TC_{\frac{1}{\lambda}}((\cdot)-\tht)\right)$  from $L^r((0,T)\times\Omega;\R)$ to $L^{2r/(2-r)}((0,T)\times\Omega;\R)$, we obtain that for any $\epsilon_2>0$ there exists $\delta_2>0$ such that if $\norm{\ths\r}_{L^r((0,T)\times\Omega))}\leq\delta_2$ it holds that
\begin{equation}\label{eq:eps_2}
\norm{\left(f\left(\TC_{\frac{1}{\lambda}}(\ths_1+\tilde{\theta})\right)- f\left(\TC_{\frac{1}{\lambda}}(\ths_2+\tilde{\theta})\right)\right)}_{L^{2r/(2-r)}((0,T)\times\Omega)}\norm{\dyw u_{2,t}}_{L^2((0,T)\times\Omega)}\leq\epsilon_2.
\end{equation}
Moreover from properties of cut-off function $\TC_{\frac{1}{\lambda}}$ and the estimate \eqref{eq:eps_1} we conclude that \eqref{eq:f1-f2} gives us
\begin{equation}\label{eq:f1-f2_2}
\begin{split}
\norm{-f\left(\TC_{\frac{1}{\lambda}}(\ths_1+\tilde{\theta})\right)\dyw u_{1,t}+ f\left(\TC_{\frac{1}{\lambda}}(\ths_2+\tilde{\theta})\right)&\dyw u_{2,t}}_{L^r((0,T)\times\Omega)}\\[1ex]
%
%
\leq& T\sup_{-1/\lambda\leq\xi\leq 1/\lambda}|f(\lambda)||\Omega|\epsilon_1+\epsilon_2.
\end{split}
\end{equation}
Now let us focus on the second term on the right-hand side of the inequality \eqref{eq:th_r1}. Since cut-off function $\TC_{\frac{1}{\lambda}}$ defines a continuous operator from $L^1((0,T)\times\Omega;\R)$ to $L^r((0,T)\times\Omega;\R)$, we firstly consider the following difference:
\begin{equation}\label{eq:T1-T2}
\begin{split}
\norm{\epsp_{1,t}:\T_1  - &\epsp_{2,t}:\T_2}_{L^1((0,T)\times\Omega)}\\[1ex]
%
%
\leq& \norm{\epsp_{1,t}:\T\r}_{L^1((0,T)\times\Omega)}+\norm{\epsp_{\Delta,t}:\T_2}_{L^1((0,T)\times\Omega)}\\[1ex]
%
%
\leq& \norm{\epsp_{1,t}}_{L^2((0,T)\times\Omega)}\norm{\T\r}_{L^2((0,T)\times\Omega)}+\norm{\epsp_{\Delta,t}}_{L^2((0,T)\times\Omega)}\norm{\T_2}_{L^2((0,T)\times\Omega)}\\[1ex]
%
%
\leq& \epsilon_1\left(\norm{\epsp_{1,t}}_{L^2((0,T)\times\Omega)}+\norm{\T_2}_{L^2((0,T)\times\Omega)}\right)\\[1ex]
%
%
\leq& \epsilon_1 C(T)\left(\sup_{-1/\lambda\leq\xi\leq 1/\lambda}|f(\lambda)| +\norm{\ve^{p,0}}_{L^2(\Omega)}+ \norm{u_0}_{\Ho}+\norm{F}_{L^2((0,T)\times\Omega)}\right),
\end{split}
\end{equation}
where we have used estimates \eqref{eq:eps_1} and \eqref{eq:ap1a}. Next fix $\epsilon_3>0$ and choose such $\epsilon_1>0$ (and consequently $\delta_1$) that the therm $\norm{\epsp_{1,t}:\T_1  - \epsp_{2,t}:\T_2}_{L^1((0,T)\times\Omega)}$ is small enough to claim that:
\begin{equation}\label{eq:eps_3}
\norm{\TC_{\frac{1}{\lambda}}\big(\epsp_{1,t}:\T_1\big)- \TC_{\frac{1}{\lambda}}\big(\epsp_{2,t}:\T_2\big)}_{L^r((0,T)\times\Omega)}\leq\epsilon_3.
\end{equation}
The estimate \eqref{eq:th_r1} together with \eqref{eq:f1-f2_2} and \eqref{eq:eps_3} give us
\begin{equation}\label{eq:th_r2}
\begin{split}
\norm{\theta\r}_{W^{1,r}(0,T;L^r(\Omega))}+ \norm{\theta\r}&_{L^r(0,T;W^{2,r}(\Omega))}\\[1ex]
&\leq C(T)\left(T\sup_{-1/\lambda\leq\xi\leq 1/\lambda}|f(\lambda)||\Omega|\epsilon_1+\epsilon_2+\epsilon_3\right)
\end{split}
\end{equation}
Therefore for any $\epsilon>0$ there exists $\delta>0$ such that if only $\norm{\ths\r}_{L^r((0,T)\times\Omega)}\leq\delta$ then
\begin{equation}
\begin{split}
\norm{\pis{A}(\ths_1)-\pis{A}(\ths_2)}_{W^{1,r}(0,T;L^r(\Omega))}+& \norm{\pis{A}(\ths_1)-\pis{A}(\ths_2)}_{L^r(0,T;W^{2,r}(\Omega))}\\[1ex]
&=\norm{\theta\r}_{W^{1,r}(0,T;L^r(\Omega))}+ \norm{\theta\r}_{L^r(0,T;W^{2,r}(\Omega))}\leq\epsilon.
\end{split}
\end{equation}
\end{proof}
%
%
\begin{lem}The operator 
$$
\pis{A}\colon L^r((0,T)\times\Omega;\R)\to L^r((0,T)\times\Omega;\R)
$$
is compact.
\label{lem:A-comp}
\end{lem}
The Lemma above is a simple consequence of the Aubin-Lions lemma (see for example \cite[Lemma 7.7]{Roubicekbook}) {\it i.e.}
$$
L^r(0,T;W^{2,r}(\Omega;\R)\cap W^{1,r}(0,T;L^r(\Omega;\R))\subset\subset L^r((0,T)\times\Omega;\R)
$$
and the Lemma \ref{lem:A-cont}.
%
%
\begin{lem} There exists $\theta\in L^r((0,T)\times\Omega;\R)$ such that $\pis{A}(\theta)=\theta$.
\label{lem:p-staly}
\end{lem}
\begin{proof}
We use the Schauder fixed point theorem. Observe that compactness of the operator $\pis{A}\colon L^r((0,T)\times\Omega;\R)\to L^r((0,T)\times\Omega;\R)$ is given by Lemma \ref{lem:A-comp}. It remains to prove that $\pis{A}$ takes its values in a ball $B(0,R)\subset L^r((0,T)\times\Omega;\R)$ for some $R>0$. Let us fix any $\ths\in L^r((0,T)\times\Omega;\R)$. Using inequalities \eqref{eq:ap2a} and \eqref{eq:ap1a} we estimate $\theta:=\pis{A}(\ths)$ as follows
\begin{equation}
\begin{split}
\norm{\theta}&_{L^r((0,T)\times\Omega)}\\[1ex]
&\leq \norm{\theta}_{W^{1,r}(0,T;L^r(\Omega))}+ \norm{\theta}_{L^r(0,T;W^{2,r}(\Omega))}\\[1ex]
&\leq  C(T)\left(\sup_{-1/\lambda\leq\xi\leq 1/\lambda}|f(\lambda)| +\norm{\ve^{p,0}}_{L^2(\Omega)}+ \norm{u_0}_{\Ho}+\norm{F}_{L^2((0,T)\times\Omega)}\right)
\end{split}
\end{equation}
and the right-hand side of the last inequality defines us a proper radius $R$.
\end{proof}
\begin{tw} For each $\lambda>0$ the system (\ref{eq:Appr}) with initial and boundary data (\ref{eq:Ini-Bond}) has a solution $(\ul,\Tl,\epspl,\thl)\in L^\infty((0,T);H^1(\Omega;\R^3)\times (L^2(\Omega;\S))^2\times W^{1,r}(\Omega;\R)$. Moreover it holds that $(\ul_t,\Tl_t,\epspl_t,\thl_t)\in L^\infty((0,T);L^2(\Omega;\R^3)\times (L^2(\Omega;\S))^2)\times L^r((0,T)\times\Omega;\R)$.
\label{tw:istn-lambda-r}
\end{tw}
\begin{proof}The assertion of Theorem is a straightforward conclusion of Lemmas from \ref{lem:ap1} to \ref{lem:p-staly}.
\end{proof}
Now we are ready to justify Theorem \ref{tw:istn-lambda}.
\begin{proof}[Proof of the Theorem \ref{tw:istn-lambda}] Since Theorem \ref{tw:istn-lambda-r} holds true, it is sufficient to prove the higher integrability of the function $\thl$. But the right-hand side of the last equation in the system \eqref{eq:Appr} ({\it i.e.} $-f\left( \TC_{\frac{1}{\lambda}}(\thl+\tilde{\theta})\right)\dyw  \ul_t+\TC_{\frac{1}{\lambda}}\left(\epspl_t: \Tl\right)$) belongs to $L^2((0,T)\times\Omega;\R)$ while the initial data $\theta_0\in H^1(\Omega)$, we conclude that $\thl$ satisfy 
$$
\theta\in L^2(0,T;H^2(\Omega;\R)),\quad\theta_t\in L^2((0,T)\times\Omega;\R),\quad\theta\in C([0,T];H^1(\Omega;\R)).
$$
This observation finishes the proof.
\end{proof}
\section{Boundedness of approximate solutions}
In this section we are going to prove the main estimates of this article. We show uniform boundedness of approximate solutions. The first estimate is similar to the energy estimate  presented in \cite[Theorem 4.1]{ChelminskiOwczarekthermoII}.
\setcounter{equation}{0}
\begin{tw}
\label{tw:4.1}
For any fixed positive number $K>0$ there exists a positive constant $C(T)$ (not depending on $\lambda$) such that the following inequality
\begin{equation*}\begin{split}
\int\limits _{\Omega}|\Tl(t)|^2\,\di x + \int\limits _0^t\int\limits _{\Omega}|\ve(u_t^{\lambda}(\tau))|^2&\,\di x\,\di\tau +\int\limits _{\Omega}|\theta^{\lambda}(t)|\,\di x\\[1ex]
&+\int\limits _0^t\int\limits _{\Omega}|\nabla\,\TC_K(\theta^{\lambda}(\tau))|^2\,\di x\,\di\tau\,\,\leq\,\, C(T)
\end{split}
\end{equation*}
is satisfied, where $t\leq T$.
\end{tw}
Proof of Theorem \ref{tw:4.1} goes along the same line as the proof of Theorem 4.1 in \cite{ChelminskiOwczarekthermoII}, hence we skip them. It is sufficient to observe the following inequality
\begin{equation*}\begin{split}
\ve^{p,\lambda}_t: \Tl=\frac{1}{2\lambda}(|P\Tl|-k)_{+}\frac{P\Tl}{|P\Tl|}: P\Tl=\frac{1}{2\lambda}(|P\Tl|-k)_{+}|P\Tl|\geq 0\,,
\end{split}
\end{equation*}
which is a consequence of the second law of thermodynamic. The same technique of proof of Theorem \ref{tw:4.1} has been also used in \cite{BlanchardGuibe00}.

The next goal is to prove $L^2-L^2$ estimate for the time derivative $\{\Tl_t\}_{\lambda>0}$. The proof of the theorem below works for gradient flows only.
\begin{tw}[$L^2 - L^2$ estimate for the stress rate] The sequence $\{\Tl_t\}_{\lambda>0}$  is uniformly (with respect to $\lambda$) bounded in the space $L^2(0,T;L^2(\Omega;\S))$.
\label{tw:4.2}
\end{tw}
\begin{proof}
Define
\begin{equation}
M_{\lambda}(x,t)=\frac{(|P\Tl(x,t)|-k)_{+}^2}{4\lambda}\,.
\label{eq:4.1}
\end{equation}
Computing the time derivative of (\ref{eq:4.1}), using equation $(\ref{eq:Appr})_3$ and $(\ref{eq:Appr})_4$ we obtain
\begin{equation}\begin{split}
\frac{\di}{\di t}\int\limits _{\Omega}M_{\lambda}\,\di x&=\int\limits _{\Omega}\frac{1}{2\lambda}(|P\Tl|-k)_{+}\frac{P\Tl}{|P\Tl|}: P\Tl_t\,\di x=\int\limits _{\Omega}\ve^{p,\lambda}_t: \Tl_t\,\di x\\[1ex]
&= -\int\limits _{\Omega}\D^{-1} \Tl_t: \Tl_t\,\di x +\int\limits _{\Omega}\ve(u^{\lambda}_t):\Tl_t\,\di x\,.
\end{split}
\label{eq:4.2}
\end{equation}
Integrating with respect to time and using assumption on initial data ($|P\Tl(0)|\leq k$) we have
\begin{equation}
\int\limits _{\Omega}M_{\lambda}\,\di x+\int\limits _0^t\int\limits _{\Omega}\D^{-1} \Tl_t \Tl_t\,\di x\di\tau =\int\limits _0^t\int\limits _{\Omega}\ve(u^{\lambda}_t)\Tl_t\,\di x\di\tau\,.
\label{eq:4.3}
\end{equation}
Theorem \ref{tw:4.1} yields that the sequence $\{\ve(u^{\lambda}_t)\}_{\lambda>0}$  is uniformly bounded (with respect to $\lambda$) in the space $L^2(0,T;L^2(\Omega;\S))$, hence the H\"older and Cauchy inequalities finish the proof.
\end{proof}
\begin{uwa}
Result from Theorem \ref{tw:4.2} will be needed to prove the strong $L^2-$ convergence of the sequence $\{\Tl\}_{\lambda>0}$ and allow us to characterise the weak limit of the field $Y_\lambda$. It is worth to mentioned that in the previous paper \cite{ChelminskiOwczarekthermoII} we used the Minty-Browder monotone trick to identify the weak limits in nonlinearities. Minty-Browder trick does not require the above information. 
\end{uwa}
\begin{col}
\label{col:4.3}
Theorems \ref{tw:4.1} and \ref{tw:4.2} imply that the sequences $\{\ve^{p,\lambda}_t\}_{\lambda>0}$ and \\$\{\ve^{p,\lambda}_t: \Tl\}_{\lambda>0}$ are uniformly (with respect to $\lambda$) bounded in the space $L^2(0,T;L^2(\Omega;\SS))$ and $L^1(0,T;L^1(\Omega;\R))$, respectively.
\end{col}
Corollary \ref{col:4.3} yields that the right-hand side of heat equation $(\ref{eq:Appr})_5$ is bounded in $L^1(0,T;L^1(\Omega;\R))$. It is sufficient to use Boccardo and Gallou\"et approach  and prove the almost pointwise convergence of temperature's approximate sequence to a measurable function $\theta$. 
\begin{lem}
\label{lem:4.4}
The sequence $\{\theta^{\lambda}\}_{\lambda>0}$ is uniformly bounded (with respect to $\lambda$) in the space $L^{q}(0,T;W^{1,q}(\Omega;\R))$ for all $1\leq q<\frac{5}{4}$.
\end{lem}
\begin{proof} We give a sketch of the proof, which is standard and can be found in several papers (see for example \cite{bocc89,GwiazdaKlaweSwierczewska14,ChelminskiOwczarekthermoII}).\\
{\it Sketch of the proof:}\\
 1.\hspace{1ex}  Assume that $q<\frac{5}{4}$ and $2\alpha'=\frac{4}{3}q$, then using Boccardo and Gallou\"et approach we arrive the following inequality
\begin{equation}\begin{split}
\int\limits _0^{T}\int\limits _{\Omega}|\nabla\theta^{\lambda}|^q\di x\,\di t\leq C 
+D\,\Big(1+\nu\int\limits _0^{T}\int\limits _{\Omega}|\theta^{\lambda}|^{2\alpha}\di x\,\di t\Big)^{\frac{q}{2}} \Big(\int\limits _0^{T}\int\limits _{\Omega}|\theta^{\lambda}|^{2\alpha'}\di x\,\di t\Big)^{1-\frac{q}{2}}\,,
\end{split}
\label{eq:4.4}
\end{equation}
where $\alpha\in (\frac{1}{2},\frac{5}{6})$ is from the growth assumption on the function $f$.\\[1ex] 
2.\hspace{1ex} The interpolation theorem yields 
\begin{equation}
\label{eq:4.5}
\|\theta^{\lambda}(t)\|_{L^{2\alpha'}(\Omega)} \leq \|\theta^{\lambda}(t)\|^{s}_{L^{1}(\Omega)}\|\theta^{\lambda}(t)\|^{1-s}_{L^{q^{\ast}}(\Omega)}
\end{equation}  
for almost every $t\leq T$, where $q^{\ast}=\frac{3q}{3-q}$ and $\frac{1}{2\alpha'}=\frac{s}{1}+\frac{1-s}{q^{\ast}}$. From Theorem \ref{tw:4.1} we conclude that the sequence $\{\theta^{\lambda}\}_{\lambda>0}$ is bounded in $L^{\infty}(0,T;L^{1}(\Omega;\R))$, thus simple calculations lead to
\begin{equation}
\label{eq:4.6}
\int\limits _0^{T}\int\limits _{\Omega}|\theta^{\lambda}|^{2\alpha'}\,\di x\di t\leq D_1\|\theta^{\lambda}\|^{q}_{L^{q}(0,T;L^{q^{\ast}}(\Omega))}\,,
\end{equation}
where the constant $D_1>0$ does not depend on $\lambda$.\\[1ex]
3.\hspace{1ex} Again the interpolation theorem implies the following inequality
\begin{equation}
\label{eq:4.7}
\|\theta^{\lambda}\|^{q}_{L^{q}(0,T;L^{q}(\Omega))}\leq D_2\Big(\int\limits _0^{T}\|\theta^{\lambda}\|^{q}_{L^{q^{\ast}}(\Omega)}\,\di t\Big)^{\frac{(q-1)g^{\ast}}{(q^{\ast}-1)q}}\,,
\end{equation}
 where $q^{\ast}=\frac{3q}{3-q}$, the exponent $\frac{(q-1)g^{\ast}}{(q^{\ast}-1)q}=\frac{3(q-1)}{3(q-1)+q}$ is less then one and the constant $D_2>0$ does not depend on $\lambda$.\\[1ex]
4.\hspace{1ex} Let $\frac{4}{3}\leq 2\alpha<\frac{5}{3}$ such that $2\alpha=\frac{4}{3}q$. Using the interpolation inequality we obtain  
\begin{equation}
\label{eq:4.8}
\|\theta^{\lambda}(t)\|_{L^{2\alpha}(\Omega)} \leq \|\theta^{\lambda}(t)\|^{s_1}_{L^{1}(\Omega)}\|\theta^{\lambda}(t)\|^{1-s_1}_{L^{q^{\ast}}(\Omega)}
\end{equation} 
for almost every $t\leq T$, where $q^{\ast}=\frac{3q}{3-q}$ and $\frac{1}{2\alpha}=\frac{s_1}{1}+\frac{1-s_1}{q^{\ast}}$. Therefore 
\begin{equation}
\label{eq:4.9}
\int\limits _0^{T}\int\limits _{\Omega}|\theta^{\lambda}|^{2\alpha}\,\di x\,\di t\leq D_3\|\theta^{\lambda}\|^{q}_{L^{q}(0,T;L^{q^{\ast}}(\Omega))}\,,
\end{equation}
where the constant $D_3>0$ does not depend on $\lambda$. Applying the Sobolev embedding theorem and combining (\ref{eq:4.4}) with (\ref{eq:4.6}), (\ref{eq:4.7}) and (\ref{eq:4.9}) we deduce the following inequality
\begin{equation}\begin{split}
\|\theta^{\lambda}\|^{q}_{L^{q}(0,T;L^{q^{\ast}}(\Omega))}\leq & D_4\int\limits _0^{T}\big(\|\theta^{\lambda}\|^{q}_{L^{q}(\Omega)}+\|\nabla\theta^{\lambda}\|^{q}_{L^{q}(\Omega)}\big)\di t\\[1ex]
\leq & D_2\Big(\int\limits _0^{T}\|\theta^{\lambda}\|^{q}_{L^{q^{\ast}}(\Omega)}\,\di t\Big)^{\frac{(q-1)g^{\ast}}{(q^{\ast}-1)q}}+C\\[1ex]
+& D\,\Big(1+\nu D_3\|\theta^{\lambda}\|^{q}_{L^{q}(0,T;L^{q^{\ast}}(\Omega))}\Big)^{\frac{q}{2}} \Big(D_1\|\theta^{\lambda}\|^{q}_{L^{q}(0,T;L^{q^{\ast}}(\Omega))}\Big)^{1-\frac{q}{2}}
\end{split}
\label{eq:4.10}
\end{equation}
The Young inequality finishes the proof in this case.\\[1ex] 
$4^{\ast}$. Assume that $1<2\alpha<\frac{4}{3}$, then 
\begin{equation}
\label{eq:4.11}
\|\theta^{\lambda}(t)\|_{L^{2\alpha}(\Omega)} \leq D_5\|\theta^{\lambda}(t)\|_{L^{\frac{4}{3}}(\Omega)}
\end{equation} 
for almost every $t\leq T$. Inequality (\ref{eq:4.11}), part $3$ and $4$ imply that the sequence $\{\theta^{\lambda}\}_{\lambda>0}$ is bounded in the space $L^{2\alpha}(0,T;L^{2\alpha}(\Omega;\R))$ for $1<2\alpha<\frac{4}{3}$. Applying this information in (\ref{eq:4.4}) we complete the proof.
\end{proof}
\begin{col}
\label{col:4.5}
The sequence $\{\theta^{\lambda}\}_{\lambda>0}$ is bounded in the space $L^q(0,T;W^{1,q}(\Omega;\R))$ for $1\leq q<\frac{5}{4}$. Moreover, the growth condition of $f$ implies that 
$$\int\limits _0^t\int\limits _{\Omega}|f\big(\TC_{\frac{1}{\lambda}}(\theta^{\lambda}+\tilde{\theta})\big)|^2\,\di x\,\di\tau\leq A+\tilde{M} \int\limits _0^t\int\limits _{\Omega}|\theta^{\lambda}|^{2\alpha}\,\di x\di\tau\,,$$
where the constants $A$ and $\tilde{M}$ do not depend on $\lambda>0$. From the proof of lemma \ref{lem:4.4} we conclude that the sequence $\{f\big(\TC_{\frac{1}{\lambda}}(\theta^{\lambda}+\tilde{\theta})\big)\}$ is bounded in $L^2(0,T;L^2(\Omega;\R))$. Hence the sequence 
$$
\{f\big(\TC_{\frac{1}{\lambda}}(\theta^{\lambda}+\tilde{\theta})\big) \mathrm{div}\,u^{\lambda}_t + \TC_{\frac{1}{\lambda}}\left(\ve^{p,\lambda}_t: \Tl\right)\}_{\lambda>0}
$$
is bounded in $L^1(0,T;L^1(\Omega;\R))$
It is also bounded in $L^1\big(0,T;\big(W^{1,q'}(\Omega;\R)\big)^{\ast}\big)$, where $\frac{1}{q}+\frac{1}{q'}=1$ and the space $\big(W^{1,q'}(\Omega;\R)\big)^{\ast}$ denotes the space of all linear bounded functionals on $W^{1,q'}(\Omega;\R)$. This two informations yield that the sequence $\{\theta^{\lambda}_t\}_{\lambda>0}$ is bounded in $L^1\big(0,T;\big(W^{1,q'}(\Omega;\R)\big)^{\ast}\big)$, hence it is relatively compact in $L^1(0,T;L^1(\Omega;\R))$ by the compactness Aubin-Lions Lemma. It contains a subsequence (again denoted using the superscript $\lambda$) such that $\theta^{\lambda}\rightarrow \theta$ a.e. in $\Omega\times(0,T)$. 
The continuity of $f$ implies that
$$f\big(\TC_{\frac{1}{\lambda}}(\theta^{\lambda}+\tilde{\theta})\big)\rightarrow f(\theta+\tilde{\theta})\quad \mathrm{a.e.\, in}\quad \Omega\times(0,T)\,.$$
Lemma \ref{lem:4.4} gives the boundedness of the sequence $\{\theta^{\lambda}\}_{\lambda>0}$ in the space $L^{p}(0,T;L^{p}(\Omega;\R))$ for $p<\frac{5}{3}$. Let us choose $r\in\R$ such that $2\alpha<r<\frac{5}{3}$, therefore the growth condition on $f$ yields that the sequence $\left\{f\left(\TC_{\frac{1}{\lambda}}(\theta^{\lambda}+\tilde{\theta})\right)\right\}_{\lambda>0}$ is bounded in $L^{\frac{r}{\alpha}}(\Omega\times(0,T);\R)$. Noticing that $\frac{r}{\alpha}>2$, we conclude from equi-integrability that 
$$
f\big(\TC_{\frac{1}{\lambda}}(\theta^{\lambda}+\tilde{\theta})\big)\rightarrow f(\theta+\tilde{\theta})\quad \mathrm{in}\quad L^2(0,T;L^2(\Omega;\R))\,.
$$
\end{col}
Observe that Theorem \ref{tw:4.1} implies that the sequence (subsequence, if necessary) $\{\Tl\}_{\lambda>0}$ is weakly convergent in $L^{\infty}(0,T;L^2(\Omega;\S))$. The last above information allows to improve this convergence.
\begin{tw}[Strong convergence of stresses]
\noindent
\label{tw:4.6}
Let us assume that the given data satisfy all requirements of Theorem \ref{tw:1.2}. Then
$$
\int\limits _{\Omega}\D^{-1}(T^{\lambda}-T^{\mu}):(T^{\lambda}-T^{\mu})\,\di x\to 0
$$
when $\lambda, \mu\to 0^+$ uniformly on bounded time intervals. 
\end{tw}
\begin{proof}
Compute the time derivative
\begin{equation}\begin{split}
&\frac{d}{\di t  }\Big(\frac{1}{2}\int\limits _{\Omega}\D\big(\ve(u^{\lambda})-\ve(u^{\mu}) -(\ve^{p,\lambda}-\ve^{p,\mu})\big):\big( \ve(u^{\lambda})-\ve(u^{\mu}) -(\ve^{p,\lambda}-\ve^{p,\mu})\big)\,\di x\Big)=\\[1ex] 
&= \int\limits _{\Omega}\D\big(\ve(u^{\lambda})-\ve(u^{\mu}) -(\ve^{p,\lambda}-\ve^{p,\mu})\big):\big(\ve(u^{\lambda}_t)-\ve(u^{\mu}_t) -(\ve^{p,\lambda}_t-\ve^{p,\mu}_t)\big)\,\di x =\\[1ex]
&=\int\limits _{\Omega}\D\big(\ve(u^{\lambda})-\ve(u^{\mu}) -(\ve^{p,\lambda}-\ve^{p,\mu})\big):\big( \ve(u^{\lambda}_t)-\ve(u^{\mu}_t)\big)\,\di x\\[1ex]
&-\int\limits _{\Omega}\big(Y_{\lambda}(\Tl)-Y_{\mu}(\T^{\mu})\big): \big(
\T^{\lambda}-\T^{\mu}\big)\,\di x\\[1ex]
&=\int\limits _{\Omega}\big(\sigma^{\lambda}-\sigma^{\mu}+\D(\ve(u^{\lambda}_t)-\ve(u^{\mu}_t))\big): \big(\ve(u^{\lambda}_t)-\ve(u^{\mu}_t)\big)\,\di x\\[1ex]
&-\int\limits _{\Omega}\big(\D(\ve(u^{\lambda}_t)-\ve(u^{\mu}_t))\big): \big(\ve(u^{\lambda}_t)-\ve(u^{\mu}_t)\big)\,\di x\\[1ex]
&-\int\limits _{\Omega} \big(f\big(\TC_{\frac{1}{\lambda}}(\theta^{\lambda}+\tilde{\theta})\big)-
f\big(\TC_{\frac{1}{\mu}}(\theta^{\mu}+\tilde{\theta})\big)\big) \big(\mathrm{div}\,u^{\lambda}_t-\mathrm{div}\,u^{\mu}_t\big)\,\di x \\[1ex]
&-\int\limits _{\Omega}\big(Y_{\lambda}(\Tl)-Y_{\mu}(\T^{\mu})\big): \big(
\T^{\lambda}-\T^{\mu}\big)\,\di x\,.
\end{split}
\label{eq:4.12}
\end{equation}
Using the fact that the given data for two approximation steps are equal and integrating with respect to time, we conclude that
\begin{equation}\begin{split}
&\frac{1}{2}\int\limits _{\Omega}\D\big(\ve(u^{\lambda})-\ve(u^{\mu}) -(\ve^{p,\lambda}-\ve^{p,\mu})\big): \big( \ve(u^{\lambda})-\ve(u^{\mu}) -(\ve^{p,\lambda}-\ve^{p,\mu})\big)\,\di x\\[1ex] 
&+\int\limits _0^t\int\limits _{\Omega}\D\big(\ve(u^{\lambda}_t)-\ve(u^{\mu}_t)\big):\big( \ve(u^{\lambda}_t)-\ve(u^{\mu}_t)\big)\,\di x\,\di\tau\\[1ex]
&\leq D\int\limits _0^t\int\limits _{\Omega} |f\big(\TC_{\frac{1}{\lambda}}(\theta^{\lambda}+\tilde{\theta})\big)-
f\big(\TC_{\frac{1}{\mu}}(\theta^{\mu}+\tilde{\theta})\big)|^2\,\di x\,\di\tau\nn\\[1ex] 
&-\int\limits _{\Omega}\big(Y_{\lambda}(\Tl)-Y_{\mu}(\T^{\mu})\big):\big(
\T^{\lambda}-\T^{\mu}\big)\,\di x\,,
\end{split}
\label{eq:4.13}
\end{equation}
where $D$ does not depend on $\lambda,\,\mu$. The Corollary \ref{col:4.5} yields that $f\big(\TC_{\frac{1}{\lambda}}(\theta^{\lambda}+\tilde{\theta})\big)\rightarrow f(\theta+\tilde{\theta})$ in  $L^2(0,T;L^2(\Omega;\R))$. Standard methods for maximal monotone operators finish the proof (See for an instance \cite{Aubin, AlbChe}).
\end{proof}
\section{Passing to the limit}
\setcounter{equation}{0}
In this section we collect all arguments to conclude with a proof of Theorem \ref{tw:1.2}.
\begin{lem}\label{lem:epsp-wl}We can choose subsequence of $\{\ve^{p,\lambda}_t\}_{\lambda>0}$ such that $\ve^{p,\lambda}_t=Y_\lambda(\T^\lambda)\rightharpoonup\chi$ when $\lambda\to 0$ in $L^2(0,T;L^2(\Omega;\S))$. Moreover $\chi\in\partial I_K(\T)$, where $\T=\lim_{\lambda\to 0}\T^\lambda$ in $L^\infty(0,T;L^2(\Omega;\S))$.
\end{lem}
\begin{proof}
By Corollary \ref{col:4.3} we choose a subsequence of $\{\ve^{p,\lambda}_t\}_{\lambda>0}$ weakly converges to $\chi$ in $L^2(0,T;L^2(\Omega;\S))$. Theorem \ref{tw:4.6} gives us that there exists $\T\in L^\infty(0,T;L^2(\Omega;\S))$ such that $\T=\lim_{\lambda\to 0}\T^\lambda$ in $L^\infty(0,T;L^2(\Omega;\S))$. By properties of Yosida approximation we have $(R_\lambda(\T^\lambda),Y_\lambda(\T^\lambda))\in \graph Y_\lambda$, where $R_\lambda(\T^\lambda)$ is a resolvent of subdifferential $\partial I_K$. By properties of the resolvent and a convergence of $\T_\lambda$ we have $R_\lambda(\T^\lambda)\to \T$ strongly in $L^\infty(0,T;L^2(\Omega;\S))$ and consequently in $L^2(0,T;L^2(\Omega;\S))$. Since a graph of the subdifferential $\partial I_K$ is strongly-weakly closed, we have $(\T,\chi)\in \partial I_K$ that is $\chi\in\partial I_K(\T)$.
\end{proof}
\begin{col} From Theorem \ref{tw:4.1} and Corollary \ref{col:4.5} it follows that by choosing subsequence we obtain:
\begin{multline}
\sigma^\lambda+\D(\ve(\ul_t))=\\
\T^\lambda+\D(\ve(\ul_t))+f\left(\TC_{\frac{1}{\lambda}}(\thl+\tht)\right)\id\rightharpoonup\T+\D(\ve(u_t))+f\left(\theta+\tht\right)\id\\
=\sigma+\D(\ve(u_t))
\end{multline}
in $L^2(0,T;L^2(\Omega;\S))$. Thus a weak divergence 
$$
\dyw\left( \sigma+\D(\ve(u_t))\right)=-F
$$
and
$$
\T=\D(\ve(u)+\epsp).
$$
Moreover Lemma \ref{lem:epsp-wl} gives us that
$$
\epsp_t\in \partial I_K(\T),
$$
where $\ve^{p,\lambda}_t\rightharpoonup \epsp_t$ when $\lambda\to 0$ in $L^2(0,T;L^2(\Omega;\S))$.
\label{wni:zbiega}
\end{col}
Now we are ready to pass to the limit with $\lambda\to+\infty$ in the approximated system \eqref{eq:Appr} and prove Theorem \ref{tw:1.2}.
\begin{proof}[Proof of Theorem \ref{tw:1.2}]
Let $\phi\in C_0^\infty([0,T);H^1(\Omega;\R)\cap L^\infty(\Omega;\R))$ be an arbitrary function and let $S\in C^\infty(\R;\R)$ be such that $\operatorname{supp}\,S'\in [-M,M]$ for some $M>0$. We use $v:=S'(\thl)\phi$ as a test function in a weak formulation of the fifth equation in the system \eqref{eq:Appr} to obtain
\begin{equation}\label{eq:ren-lambda}
\begin{split}
\int\limits _0^T&\int\limits _\Omega S(\thl)\phi_t+ \nabla\thl\nabla(S'(\thl)\phi)\,\di x\, \di t+ \int\limits_\Omega S(\theta_0)\phi(0,x)\,\di x  \\[1ex]
&=\int\limits _0^T\int\limits _\Omega\left(\TC_{\frac{1}{\lambda}}\left(\ve^{p,\lambda}_t: \Tl\right)-f\left(\TC_{\frac{1}{\lambda}}(\thl+\tht)\right)\dyw u^\lambda_t\right)S'(\thl)\phi\,\di x\, \di t  .
\end{split}
\end{equation}
Our goal is to pass with $\lambda\to 0$. First observe that exactly as in the proof of Theorem 5.2 in \cite{ChelminskiOwczarekthermoII} we prove that
$$
\int\limits _0^T\int\limits _\Omega \nabla\thl\nabla(S'(\thl)\phi)\,\di x \di t  \to \int\limits _0^T\int\limits _\Omega \nabla\theta\nabla(S'(\theta)\phi)\,\di x\, \di t\,,
$$
when $\lambda\to 0$. Next notice it holds that
$$
\TC_{\frac{1}{\lambda}}\left(\ve^{p,\lambda}_t: \Tl\right)-f\left(\TC_{\frac{1}{\lambda}}(\thl+\tht)\right)\dyw u^\lambda_t\rightharpoonup \epsp_t: \T-f\left(\theta+\tht\right)\dyw u_t  \quad\text{when }\lambda\to 0
$$
in $L^1((0,T)\times\Omega)$ when $\lambda\to 0$. It follows from Theorem \ref{tw:4.1}, Corollary \ref{col:4.3} and Corollary \ref{col:4.5}.
Let us denote
\begin{eqnarray*}
\psi^\lambda&:=&\TC_{\frac{1}{\lambda}}\left(\ve^{p,\lambda}_t: \Tl\right)-f\left(\TC_{\frac{1}{\lambda}}(\thl+\tht)\right)\dyw u_t^\lambda\\
\psi&:=&\epsp_t: \T-f\left(\theta+\tht\right)\dyw u_t
\end{eqnarray*}
and consider a convergence of the integral on the right-hand of \eqref{eq:ren-lambda}.
\begin{equation*}
\begin{split}
\int\limits _0^T\int\limits _\Omega&\psi^\lambda S'(\thl)\phi-\psi S'(\theta)\phi\,\di x \di t  \\[1ex]
=&\int\limits _0^T\int\limits _\Omega\psi^\lambda \left(S'(\thl)-S'(\theta)\right)\phi\,\di x \di t  -\int\limits _0^T\int\limits _\Omega(\psi-\psi^\lambda)S'(\theta)\phi\,\di x \di t  \\[1ex]
=& I_1+I_2.
\end{split}
\end{equation*}
Obviously $I_2\to 0$, because $\psi^\lambda\rightharpoonup \psi$ in $L^1((0,T)\times\Omega)$ when $\lambda\to 0$. Let fix arbitrary $\delta>0$. By Dunford-Pettis theorem we choose such small $\epsilon>0$ that for each $E\subset[0,T]\times\Omega$ such that $|E|<\epsilon$ we have
$$
\iint\limits _E\psi^\lambda\,\di x \di t  <\frac{\delta}{2\sup_{\xi\in[-M,M]}|S'(\xi)|},
$$
where the support of $S'$ was previously assumed to be included in a interval $[-M,M]$.
Since by Lemma \ref{lem:4.4} $\thl\to\theta$ a.e. in $[0,T)\times\Omega$ (choosing subsequence if necessary), by Egorov theorem we can choose a set $E\subset[0,T)\times\Omega$ such that $|E|<\epsilon$ and $S'(\thl)\rightrightarrows S'(\theta)$ on $([0,T]\times\Omega)\setminus E$. Finally
\begin{equation*}
\begin{split}
|I_1|&\leq\left|\phantom{z}{\iint\limits_{([0,T]\times\Omega)\setminus E}\psi^\lambda \left(S'(\thl)-S'(\theta)\right)\phi\,\di x \di t  }\right|+ \left|\iint\limits_{E}\psi^\lambda \left(S'(\thl)-S'(\theta)\right)\phi\,\di x \di t  \right|\\[1ex]
&\leq \left|\phantom{z}{\iint\limits_{([0,T]\times\Omega)\setminus E}\psi^\lambda \left(S'(\thl)-S'(\theta)\right)\phi\,\di x \di t  }\right|+\delta\underset{\lambda\to 0}{\longrightarrow} \delta.
\end{split}
\end{equation*}
Because $\delta>0$ was chosen arbitrary we obtain $\lim_{\lambda\to 0}I_1=0$.

Finally the weak-* converges in $L^\infty((0,T)\times\Omega)$ of $S(\thl)$ to $S(\theta)$ gives us
$$
\int\limits _0^T\int\limits _\Omega S(\thl)\phi_t\,\di x \di t  \to \int\limits _0^T\int\limits _\Omega S(\theta)\phi_t\,\di x \di t  .
$$
Therefore we pass to the limit with $\lambda\to 0$ in \eqref{eq:ren-lambda} and obtain:
\begin{equation}\label{eq:renolm}
\begin{split}
-\int\limits _0^T&\int\limits _\Omega S(\theta)\phi_t+ \nabla\theta\nabla(S'(\theta)\phi)\,\di x \di t + \int\limits_\Omega S(\theta_0)\phi(0,x)\,\di x   \\[1ex]
-\int\limits _0^T&\int\limits _\Omega S(\theta)\phi_t+ \nabla\theta S'(\theta)\nabla\phi + |\nabla\theta|^2S''(\theta)\phi\,\di x \di t + \int\limits_\Omega S(\theta_0)\phi(0,x)\,\di x   \\[1ex]
&=\int\limits _0^T\int\limits _\Omega\left(\epsp_t: \T-f\left(\theta+\tht\right)\right)S'(\thl)\phi\,\di x \di t  .
\end{split}
\end{equation}
The equality \eqref{eq:renolm} together with Corollary \ref{wni:zbiega} end the proof.
\end{proof}
%
%
\appendix
\section{Linear elasticity} In this section we prove an existence of the auxiliary problem in linear elasticity.
\begin{lem}[Problem of linear elasticity] Let $b\in L^2(0,T;(H^1(\Omega;\R^3)^{\ast})$ while $u_0\in H^1(\Omega;\R^3)$ and $g_D\in L^2(0,T;H^{1/2}(\partial\Omega;\R^3))$. Then there exists a unique solution $u\in H^1(0,T;H^1(\Omega;\R^3))$ to the following problem
\begin{equation}
\begin{array}{rlll}
-\dyw\D(\ve(u(x,t)))-\dyw\D(\ve(u_t(x,t)))&=b(x,t),\quad &\text{for}&\;(x,t)\in(0,T)\times\Omega,\\
u(x,t)&=g_D,& \text{for}&\; (x,t)\in(0,T)\times\partial\Omega,\\
u(0,x)&=u_0(x),&\text{for}&\; x\in\Omega.
\end{array} \label{eq:pom1}
\end{equation}
Furthermore the solution $u$ can be estimated as follows:
\begin{equation}
\norm{u}_{H^1(0,T;\Ho)}\leq C(T)\left(\norm{u_0}_{\Ho}+\norm{g_D}_{L^2(0,T;H^{1/2}(\partial\Omega))}+ \norm{b}_{L^2(0,T;(H^1(\Omega)^{\ast})}\right),
\label{eq:LEestim}
\end{equation}
where a constant $C(T)>0$ depends only on a length of a time interval, a geometry of the domain $\Omega$ and entries of the operator $\D$. 
\label{lem:LE}
\end{lem}
\begin{proof}
The proof goes the similar way as a proof of Collorary 1 in \cite{BartOwcz1}. 
We use the Banach fixed point theorem, thus we construct a contractive operator
$$
\pis{P}\colon L^{2}(0,T;\HoR)\to L^{2}(0,T;\HoR).
$$
Let $v\in L^{2}(0,T;\HoR)$ and we look for the solution $w$ to the following problem:
\begin{equation*}
\begin{array}{rlll}
-\dyw\D(\ve(w))&=\dyw\D(\ve(v))+b(x,t),\quad &\text{in}&\;(0,T)\times\Omega,\\
w&=g_D,& \text{in}&\; (0,T)\times\partial\Omega,\\
w|_{t=0}&=u_0,&\text{in}&\; \Omega.
\end{array} \tag{$\star$}
\end{equation*}
We uniquely solve the problem ($\star$) as a straightforward conclusion from the ellipticity of the operator $-\dyw\D\ve(\cdot)$ and we obtain that the solution of $(\star)$ $w$ is of the class $L^2(0,T;\HoR)$. Thus we put $u(x,t)=\pis{P}(v)(x,t):=\int\limits _0^t w(\tau,x)\,\di \tau   +u_0(x)$. Obviously the operator $\pis{P}\colon L^{2}(0,T;\HoR)\to L^{2}(0,T;\HoR)$ is well defined. Now we insert $v_1,v_2\in L^2(0,T;\HoR)$ into the system ($\star$) and obtain the solutions $w_1,w_2\in L^{2}(0,T;\HoR)$. Therefore the difference $w_1-w_2$ satisfies the following system:
\begin{equation*}
\begin{array}{rlll}
-\dyw\D(\ve(w_1-w_2))&=\dyw\D(\ve(v_1-v_2)),\quad &\text{in}&\;(0,T)\times\Omega,\\
w_1-w_2&=0,& \text{in}&\; (0,T)\times\partial\Omega,\\
(w_1-w_2)|_{t=0}&=0,&\text{in}&\; \times\Omega.
\end{array} \tag{$\star\star$}
\end{equation*}
Then a standard elliptic estimates conclude with
\begin{equation}
\begin{split}
\norm{u_1-u_2}_{L^{2}(0,t;\Ho)}
&\leq Ct\norm{v_1-v_2}_{L^{2}(0,t;\Ho)}
\label{eq:1stbanach}
\end{split}
\end{equation}
and obviously the operator $\pis{P}_1\colon L^{2}(0,T_1;\HoR)\rightarrow L^{2}(0,T_1;\HoR)$ is a contraction for $T_1=\frac{1}{2C}$. Thus by the Banach fixed point theorem we obtain an existence of an unique fix point $u\in L^{2}(0,T_1;\HoR)$ of $\pis{P}$ {it i.e.} $\pis{P}_1(u)=u$. Additionally the construction of the operator $\pis{P}$ gives us  immediately  that $u_t\in L^{2}(0,T_1;\HoR)$. Moreover the estimate (\ref{eq:1stbanach}) does not depend on the initial condition thus by repetition of the reasoning above we obtain a sequence of contractive operators $$\pis{P}_k\colon L^{2}(T_{k-1},T_k;\HoR)\rightarrow L^{2}(T_{k-1},T_k;\HoR),$$ where $T_k=\frac{k}{2C}$ and corresponding solutions  $u\in H^1(T_{k-1},T_k;\HoR)$.
Next we prove the estimate (\ref{eq:LEestim}). For this purpose we integrate the system (\ref{eq:pom1}) with respect to time and obtain:
\begin{equation*}
-\dyw\D(\ve(u(x,t)))=-\dyw\D(\ve(u_0(x)))+\int\limits _0^t \dyw\D(\ve(u(\tau,x)))\,\di \tau   + \int\limits _0^t b(\tau,x)\,\di \tau   \,.
\end{equation*}
Once again using ellipticity of the operator $-\dyw\D\ve(\cdot)$, for any $t\in(0,T) $ we obtain the following estimate:
\begin{equation*}\begin{split}
\norm{u(t)}_{\Ho}\leq C\Big(& \norm{u_0}_{\Ho}+ \int\limits _0^t\norm{u(\tau)}_{\Ho}\di \tau   \\
&+ \int\limits _0^t\norm{b(\tau)}_{(H^1(\Omega))^\ast}\di \tau   + \int\limits _0^t\norm{g_D(\tau)}_{H^{1/2}(\partial\Omega)}\di \tau   \Big)\,.
\end{split}
\end{equation*}
Gronwall inequality implies that:
\begin{equation*}
\norm{u(t)}_{\Ho}\leq C\left(e^{Ct}+1\right)\Big(\norm{u_0}_{\Ho}+ \int\limits _0^t\norm{b(\tau)}_{(H^1(\Omega))^\ast}+ \norm{g_D(\tau)}_{H^{1/2}(\partial\Omega)}\di \tau   \Big)\,,
\end{equation*}
and
\begin{equation*}
\norm{u}_{L^2(0,T);\Ho)}\leq C(T)\Big(\norm{u_0}_{\Ho}+ \norm{b}_{L^2(0,T;(H^1(\Omega))^\ast)}+ \norm{g}_{L^2(0,T;H^{1/2}(\partial\Omega))}\Big).
\end{equation*}
Next, using a fact that $u$ satisfies the system (\ref{eq:pom1}), we estimate a norm $\norm{u_t(t)}$ for any $t\in(0,T)$ as follows
\begin{equation*}
\norm{u_t(t)}_{\Ho}\leq C\Big(\norm{u(t)}_{\Ho}+\norm{b(t)}_{(H^1(\Omega))^\ast}+ \norm{g_D(t)}_{H^{1/2}(\partial\Omega)}\Big)
\end{equation*}
and obviously we obtain
\begin{equation*}
\norm{u_t}_{L^2(0,T;\Ho)}\leq C(T)\Big(\norm{u_0}_{\Ho}+ \norm{b}_{L^2(0,T;(H^1(\Omega))^\ast)}+ \norm{g_D}_{L^2(0,T;H^{1/2}(\partial\Omega))}\Big).
\end{equation*}
Thus the proof is completed.
\end{proof}
{\bf Acknowledgments: }
This work is supported by the Grant of the Polish National Science Center 
OPUS 4 2012/07/B/ST1/03306.
%
%

\bibliographystyle{abbrv}

\end{document}